\documentclass[a4paper]{amsart}
\usepackage[margin=30mm]{geometry}
\usepackage{microtype}

\usepackage[all]{xy}

\usepackage{amssymb, amsthm, amsmath, graphics, indentfirst, paralist, sseq, mathtools, float}

\usepackage[scr=boondox]{mathalfa}

\theoremstyle{plain}

\newtheorem{thm}{Theorem}[section]
\newtheorem{prop}[thm]{Proposition}
\newtheorem{lem}[thm]{Lemma}

\theoremstyle{definition}
\newtheorem{defin}[thm]{Definition}
\newtheorem{ex}[thm]{Example}

\theoremstyle{remark}
\newtheorem{rem}[thm]{Remark}

\begin{document}

\newcommand{\Ker}{\mathop{\mathrm{Ker}}\nolimits}
\newcommand{\Image}{\mathop{\mathrm{Im}}\nolimits}
\newcommand{\rk}{\mathop{\mathrm{rk}}\nolimits}
\newcommand{\sk}{\mathop{\mathrm{sk}}\nolimits}
\newcommand{\fr}{\mathrm{fr}}
\newcommand{\id}{\mathop{\mathrm{id}}\nolimits}
\newcommand{\ad}{\mathop{\mathrm{ad}}\nolimits}
\newcommand{\interior}{\mathop{\mathrm{int}}\nolimits}
\newcommand{\Hom}{\mathop{\mathrm{Hom}}\nolimits}
\newcommand{\Tor}{\mathop{\mathrm{Tor}}\nolimits}
\newcommand{\Ext}{\mathop{\mathrm{Ext}}\nolimits}
\newcommand{\Aut}{\mathop{\mathrm{Aut}}\nolimits}
\newcommand{\RU}{\mathop{\mathrm{RU}}\nolimits}
\newcommand{\st}{\mathop{\mathrm{st}}\nolimits}
\newcommand{\HH}{\mathcal{H}}
\newcommand{\Z}{\mathbb{Z}}
\newcommand{\R}{\mathbb{R}}
\newcommand{\C}{\mathbb{C}}
\newcommand{\UM}{\underline{M}}
\newcommand{\UMp}{\underline{M}{}'}
\newcommand{\UMpp}{\underline{M}{}''}
\newcommand{\UtM}{\underline{\tilde{M}}}
\newcommand{\UbM}{\underline{\bar{M}}}
\newcommand{\UN}{\underline{N}}
\newcommand{\UNp}{\underline{N}{}'}
\newcommand{\UH}{\underline{H}}
\newcommand{\UR}{\underline{R}}
\newcommand{\URp}{\underline{R}{}'}
\newcommand{\UE}{\underline{E}}
\newcommand{\UEs}{\underline{E}{}^*}
\newcommand{\UV}{\underline{V}}
\newcommand{\UVp}{\underline{V}{}'}
\newcommand{\UVs}{\underline{V}{}^*}
\newcommand{\UhV}{\underline{\hat{V}}}
\newcommand{\UhVs}{\underline{\hat{V}}{}^*}
\newcommand{\UY}{\underline{Y}}
\newcommand{\UZ}{\underline{Z}}
\newcommand{\UP}{\underline{P}}
\newcommand{\UPs}{\underline{P}{}^*}
\newcommand{\UbP}{\underline{\bar{P}}}
\newcommand{\UbPs}{\underline{\bar{P}}{}^*}
\newcommand{\gc}{\mathfrak{g}}
\newcommand{\gs}{\mathscr{g}}
\newcommand{\s}{\mathscr{s}}
\newcommand{\si}{\mathrel{\underset{^{\scriptstyle{s}}}{\cong}}}
\newcommand{\QQ}{\mathcal{Q}}
\newcommand{\MM}{\mathcal{M}}
\newcommand{\II}{\mathcal{I}}
\newcommand{\JJ}{\mathcal{J}}
\newcommand{\EE}{\mathcal{E}}

\newcommand{\fp}{F}
\newcommand{\tp}{T}

\newcommand{\Diff}{\mathop{\mathrm{Diff}}\nolimits}
\newcommand{\Homeo}{\mathop{\mathrm{Homeo}}\nolimits}
\newcommand{\dd}{\underline{d}}

\newcommand{\xra}{\xrightarrow}

\title{Realising automorphisms of the extended intersection form by diffeomorphisms}
\author{Csaba Nagy}
\email{nagy@mpim-bonn.mpg.de}
\address{Max Planck Institute for Mathematics, Vivatsgasse 7, 53111 Bonn, Germany}

\begin{abstract}
Suppose that $f : M \rightarrow B$ is a normal $(q{-}1)$-smoothing of a $2q$-manifold $M$ over some $(B,\xi)$, where $q$ is even and $B$ is simply-connected. A diffeomorphism of $M$ over $B$ induces an automorphism of the extended intersection form (or ``Q-form") of $f$, which consists of the intersection form of $M$ and the induced homomorphism $H_q(f)$. Assuming that $H_q(B)$ is free, we determine precisely which automorphisms can be realised by such diffeomorphisms. In particular, if $H_{q-1}(B)$ is also free, then we show that every automorphism can be realised. These results are obtained by studying a ``two-sided" version of the extended surgery obstruction, which was introduced in earlier work of the author. 

As an application, we show that for a complex $q$-dimensional complete intersection, with $q > 2$ even, every automorphism of the cohomology ring is realised by a diffeomorphism. 
\end{abstract}

\maketitle

\section{Introduction}

\subsection{Main results}

A diffeomorphism of a $2q$-manifold $M$ induces an automorphism of its intersection form, and the resulting map $\pi_0 \Diff^+(M) \rightarrow \Aut(H_q(M),\lambda_M)$ is often a key tool in studying the mapping class group of $M$. We will consider the analogous map in the context of normal $(q{-}1)$-smoothings. 

Suppose that $M$ is a smooth, closed, oriented, simply-connected $2q$-manifold for $q \geq 2$ even. Let $f : M \rightarrow B$ be a normal $(q{-}1)$-smoothing over some $(B,\xi)$, that is, a $q$-connected map covered by a map $\bar{f} : \nu_M \rightarrow \xi$ of stable bundles (for example, $(B,\xi)$ could be the normal $(q{-}1)$-type of $M$ and $f$ the canonical normal $(q{-}1)$-smoothing). Then the \emph{Q-form} (or extended intersection form) of $f$, 
\[
E_q(M,f) = (H_q(M),\lambda_M,H_q(f))
\]
consists of the intersection form of $M$ together with the induced map $H_q(f) : H_q(M) \rightarrow H_q(B)$ (see \cite[Sections 2 and 5]{csn-qfc26}). An orientation-preserving diffeomorphism $G$ of $M$ over $B$ (ie.\ one that satisfies $f \circ G \simeq f$) 
induces an automorphism of $E_q(M,f)$. Assuming that $H_q(B)$ is free, we will describe precisely which automorphisms are induced by such diffeomorphisms, that is, we determine the image of the map from diffeomorphisms over $B$ to $\Aut(E_q(M,f))$. 

More generally, if $f_i : M_i \rightarrow B$ is a normal $(q{-}1)$-smoothing for $i=0,1$, then a diffeomorphism $M_0 \rightarrow M_1$ over $B$ induces an isomorphism $E_q(M_0,f_0) \rightarrow E_q(M_1,f_1)$, and we will consider which isomorphisms can be obtained this way. Of course, it is necessary for the existence of a diffeomorphism over $B$ that $f_0$ and $f_1$ are normally bordant, which we will assume from now on. Furthermore, in order to include the $q=2$ case, we will formulate our statements using h-cobordisms instead of diffeomorphisms. If $F : W \rightarrow B$ is a normal bordism between $f_0$ and $f_1$ such that $W$ is an h-cobordism, then the isomorphism induced by $W$, denoted
\[
I_W : E_q(M_0,f_0) \rightarrow E_q(M_1,f_1),
\] 
is given by the composition $H_q(M_0) \rightarrow H_q(W) \rightarrow H_q(M_1)$ of the isomorphisms $H_q(M_i) \cong H_q(W)$ induced by the inclusions $M_i \rightarrow W$. If $q>2$, then by the h-cobordism theorem there is a diffeomorphism $M_0 \rightarrow M_1$ over $B$, which induces the isomorphism $I_W$ between the Q-forms. 

In \cite[Theorem 1.2]{csn-qfc26} we proved that if $H_q(B)$ is free and there is an isomorphism $I : E_q(M_0, f_0) \rightarrow E_q(M_1, f_1)$, then there is a normal bordism $F' : W' \rightarrow B$ between $f_0$ and $f_1$ such that $W'$ is an h-cobordism. When $H_{q-1}(B)$ is also free, our first main result strengthens this statement by adding that we can require $W'$ to induce the given isomorphism $I$.

\begin{thm} \label{thm:main-free}
Let $B$ be a simply-connected space, $\xi$ a stable bundle over it and $q \geq 2$ an even integer. Let $M_0$ and $M_1$ be two closed oriented $2q$-manifolds. Let $f_i : M_i \rightarrow B$ be a normal $(q{-}1)$-smoothing over $(B,\xi)$ ($i=0,1$), and $F_0 : W_0 \rightarrow B$ a normal bordism between $f_0$ and $f_1$. Suppose that $H_q(B)$ and $H_{q-1}(B)$ are free. 

Then for every isomorphism $I : E_q(M_0, f_0) \rightarrow E_q(M_1, f_1)$ there is a normal bordism $F' : W' \rightarrow B$, normally bordant (rel boundary) to $F_0$, such that $W'$ is an h-cobordism and $I_{W'} = I$. In particular, if $q > 2$, then $I$ is realised by a diffeomorphism $M_0 \rightarrow M_1$ over $B$.
\end{thm}

In contrast to Theorem \ref{thm:main-free}, without the assumption that $H_{q-1}(B)$ is free, not every isomorphism between the Q-forms can be realised by an h-cobordism over $B$ (in general). To describe which isomorphisms can be realised, we introduce 
\[
\EE_q(M,f), 
\]
the \emph{system of Q-forms of $f$}. This consists of $E_q(M,f)$ and its mod $m$ analogue, $E_q(M,f)_m$ for every $m \geq 2$. Here $E_q(M,f)_m = (H_q(M; \Z_m), \lambda_{M,m}, H_q(f; \Z_m))$ is the mod $m$ intersection form $\lambda_{M,m} : H_q(M; \Z_m) \times H_q(M; \Z_m) \rightarrow \Z_m$ of $M$ together with the induced map $H_q(f; \Z_m) : H_q(M; \Z_m) \rightarrow H_q(B; \Z_m)$. These satisfy obvious compatibility conditions involving the inclusion $H_q(M) \otimes \Z_m \rightarrow H_q(M; \Z_m)$ from the universal coefficient exact sequence. See Section \ref{ss:sys} for the detailed definitions. An h-cobordism (in fact, any orientation-preserving homotopy equivalence) over $B$ induces an isomorphism $\II : \EE_q(M_0, f_0) \rightarrow \EE_q(M_1, f_1)$ between the systems of Q-forms, see Proposition \ref{prop:ind-sys}. So if an isomorphism $I : E_q(M_0, f_0) \rightarrow E_q(M_1, f_1)$ can be realised by an h-cobordism, then it has to extend to such an $\II$. In general, this is a non-trivial condition on $I$, see Example \ref{ex:i-ext}. We show that this is the only requirement for the realisability of $I$. 

\begin{thm} \label{thm:main-torsion}
Let $B$ be a simply-connected space, $\xi$ a stable bundle over it and $q \geq 2$ an even integer. Let $M_0$ and $M_1$ be two closed oriented $2q$-manifolds. Let $f_i : M_i \rightarrow B$ be a normal $(q{-}1)$-smoothing over $(B,\xi)$ ($i=0,1$), and $F_0 : W_0 \rightarrow B$ a normal bordism between $f_0$ and $f_1$. Suppose that $H_q(B)$ is free. 

Let $I : E_q(M_0, f_0) \rightarrow E_q(M_1, f_1)$ be an isomorphism. Then there is a normal bordism $F' : W' \rightarrow B$ between $f_0$ and $f_1$ such that $W'$ is an h-cobordism and $I_{W'} = I$ if and only if $I$ extends to an isomorphism $\II : \EE_q(M_0, f_0) \rightarrow \EE_q(M_1, f_1)$. If this is the case, then there is an $F'$ that is normally bordant (rel boundary) to $F_0$. 
\end{thm}

\begin{rem}
When $H_q(B)$ is free and $H_{q-1}(B)$ contains torsion, the combination of \cite[Theorem 1.2]{csn-qfc26} and Theorem \ref{thm:main-torsion} means that if there exists any isomorphism $I : E_q(M_0, f_0) \rightarrow E_q(M_1, f_1)$, then $M_0$ and $M_1$ are diffeomorphic over $B$, even though $I$ may not be realisable by such a diffeomorphism (if it does not extend to an isomorphism $\II : \EE_q(M_0, f_0) \rightarrow \EE_q(M_1, f_1)$).
\end{rem}

When $I : E_q(M, f) \rightarrow E_q(M, f)$ is an automorphism, we also give a criterion for the existence of an automorphism $\II : \EE_q(M, f) \rightarrow \EE_q(M, f)$ extending $I$ in terms of a matrix representation of $I$. Write $H_q(M) = \fp \oplus \tp$, where $\fp = H_q(M) / \Tor H_q(M)$ and $\tp = \Tor H_q(M)$ denote the free and torsion parts of $H_q(M)$, respectively (and we fix some splitting $\fp \rightarrow H_q(M)$ of the quotient map). The Q-form $E_q(M, f)$ determines a canonical element 
\[
\hat{\mu} \in \fp \otimes H_q(B)
\] 
(see Definition \ref{def:can}). The automorphism $I$ can be written in the form $I = \bigl[ \begin{smallmatrix} a & 0 \\ b & c \end{smallmatrix} \bigr] : \fp \oplus \tp \rightarrow \fp \oplus \tp$ for some automorphism $a$ of the induced form on $\fp$, homomorphism $b : \fp \rightarrow \tp$ and automorphism $c$ of $\tp$. Then we have the following (see Theorem \ref{thm:real-aut-diff}): 

\begin{thm} \label{thm:b-c-intro}
Let $B$ be a simply-connected space, $\xi$ a stable bundle over it and $q \geq 2$ an even integer. Let $M$ be a closed oriented $2q$-manifold with a normal $(q{-}1)$-smoothing $f : M \rightarrow B$ over $(B,\xi)$. Suppose that $H_q(B)$ is free. 

Let $I = \bigl[ \begin{smallmatrix} a & 0 \\ b & c \end{smallmatrix} \bigr] : \fp \oplus \tp \rightarrow \fp \oplus \tp$ be an automorphism of $E_q(M,f)$. Then $I$ extends to an automorphism $\II$ of $\EE_q(M, f)$ if and only if $(b \otimes \id_{H_q(B)})(\hat{\mu}) = 0 \in \tp \otimes H_q(B)$ and $c = \id_{\tp}$.
\end{thm}

We end with an application. We consider the case of complete intersections, where the result of Theorem \ref{thm:main-free} about diffeomorphisms over $B$ can be used to obtain information about arbitrary diffeomorphisms.
Recall that a complete intersection $X_q(\dd)$, for a finite multiset $\dd = \left\{ d_1, d_2, \ldots , d_k \right\}$ of positive integers, is defined as the transverse intersection of smooth algebraic hypersurfaces in ${\C}P^{q+k}$ of degrees $d_1, d_2, \ldots , d_k$. It is a smooth manifold of real dimension $2q$, and by a result of Thom its diffeomorphism type depends only on $\dd$.

\begin{thm} \label{thm:main-ci}
If $q > 2$ is even, then every automorphism of the cohomology ring $H^*(X_q(\dd))$ can be realised by a diffeomorphism of $X_q(\dd)$. 
\end{thm}

\subsection{Background}

Determining the image of the map $A_M : \pi_0 \Diff^+(M) \rightarrow \Aut(H_q(M),\lambda_M)$ given by the induced automorphism is usually the first step in understanding the mapping class group of a $2q$-manifold $M$. In lower dimensions, $A_M$ is surjective for many manifolds. For example, if $M$ is a genus-$g$ surface, then it is well-known that its mapping class group surjects onto $\Aut(H_1(M),\lambda_M) \cong Sp(2g,\Z)$, and the kernel of $A_M$ is called the Torelli group. Wall \cite{wall64b} showed that $A_M$ is surjective for certain simply-connected $4$-manifolds stabilised by $S^2 \times S^2$. By Quinn \cite{quinn86} the analogous map $\pi_0 \Homeo^+(M) \rightarrow \Aut(H_q(M),\lambda_M)$ is an isomorphism for all simply-connected topological $4$-manifolds. For simply-connected $6$-manifolds with torsion-free homology $A_M$ is surjective by the arguments of Kreck-Su \cite{kreck-su25}, using Wall's splitting theorem \cite{wall66-v}.

If $q \geq 3$ and $M$ is an almost parallelisable $(q{-}1)$-connected $2q$-manifold, then Kreck \cite{kreck79}, based on results of Wall \cite{wall63-ii} and Kosi\'{n}ski \cite{kosinski67}, showed that $A_M$ is surjective onto those automorphisms that preserve a function $\alpha : H_q(M) \rightarrow \pi_{q-1}(SO_q)$. When $q$ is even, this means that $A_M$ is surjective, because $\alpha$ is determined by self-intersections. This case is recovered by taking $B = BO \left< q+1 \right>$ in Theorem \ref{thm:main-free}. 

However, the map $A_M$ cannot be surjective in general. A diffeomorphism of $M$ induces an automorphism of the entire cohomology ring and preserves characteristic classes, and this puts extra constraints on the automorphisms of $\lambda_M$ that can be realised (when $M$ is not highly-connected, or not $q$-parallelisable). 

The map $A_M$ is also closely related to (a certain aspect of) the problem of classifying $2q$-manifolds. In known results it is often the intersection form (possibly together with other invariants) that determines the diffeomorphism classification, see eg.\ Wall \cite{wall62} for $(q{-}1)$-connected almost closed $2q$-manifolds, $q \geq 3$, or in the topological case Freedman's classification of simply-connected $4$-manifolds up to homeomorphism \cite{freedman82}. In such cases a more refined question is whether every isomorphism $(H_q(M_0),\lambda_{M_0}) \rightarrow (H_q(M_1),\lambda_{M_1})$ of intersection forms (preserving the other invariants) is realised by a diffeomorphism $M_0 \rightarrow M_1$. Since two such isomorphisms differ by an automorphism of either $(H_q(M_i),\lambda_{M_i})$, the answer is affirmative if and only if $A_{M_i}$ is surjective (onto the subgroup of automorphisms preserving the other invariants). 

Our results address the analogous problems for normal $(q{-}1)$-smoothings, where diffeomorphisms and automorphisms/isomorphisms of the intersection form are replaced by diffeomorphisms over $B$ and automorphisms/isomorphisms of the Q-form, respectively. This is motivated by \cite[Theorem 1.2]{csn-qfc26}, where we proved that if normally bordant normal $(q{-}1)$-smoothings have isomorphic Q-forms (and $H_q(B)$ is free), then the manifolds are diffeomorphic over $B$. However, the question of realising a given isomorphism by a diffeomorphism remained open. 

Theorem \ref{thm:main-free} can be thought of as the analogue of the results about the surjectivity of $A_M$, but it applies to a much larger class of $2q$-manifolds. This is made possible by the fact that for normal $(q{-}1)$-smoothings, the previously mentioned constraints on the automorphisms of $\lambda_M$ that are realisable by diffeomorphisms are mostly absorbed into the condition that the induced map $H_q(f)$ has to be preserved. On the other hand, allowing only diffeomorphisms over $B$ puts new restrictions on realisable automorphisms, which appear in Theorems \ref{thm:main-torsion} and \ref{thm:b-c-intro}. In particular, the necessity of the condition $c = \id_{\tp}$ in the latter can be seen immediately, since such a diffeomorphism induces the identity on $H_{q-1}(M)$ and preserves the nonsingular linking pairing between $\Tor H_{q-1}(M)$ and $\Tor H_q(M)$, cf.\ Remark \ref{rem:tor}.

\subsection{Overview of proofs}

The basic idea is that we can modify the extended surgery obstruction used in \cite{csn-qfc26} in a way that allows us to read off the isomorphism induced by an h-cobordism from the obstruction associated to it. 

First, we introduce the two-sided $\ell$-set $\ell_{2q+1}(Q,v,\UP_0,\UP_1)$ in Section \ref{s:l-set}. This is based on the extended $\ell$-monoid $\ell_{2q+1}(Q,v)$ (see \cite[Section 4]{csn-qfc26}), but it also depends on a pair of extended quadratic forms $\UP_i$, which correspond to the Q-forms $E_q(M_i,f_i)$ in the geometric setting. Its elements are equivalence classes of two-sided quasi-formations, which include as part of their data a pair of morphisms $p_i$ to the extended quadratic forms $\UP_i$. If a two-sided quasi-formation is elementary (and proper), it determines an isomorphism $\UP_0 \rightarrow \UP_1$, see Definition \ref{def:ix}.

In Section \ref{s:obstr-def} we define the two-sided surgery obstruction $\theta^{f_0,f_1}_{W,F}$ of a $q$-connected normal bordism $F : W \rightarrow B$ between $f_0$ and $f_1$. It is the extended surgery obstruction $\theta_{W,F}$ from \cite[Section 6]{csn-qfc26} together with a pair of surjective morphisms $p_i : V_i \rightarrow H_q(M_i)$ that make it a (proper) two-sided quasi-formation. It inherits the main properties of the extended obstruction, in particular, its equivalence class is a normal bordism invariant, and (up to stable isomorphism) any two-sided quasi-formation equivalent to $\theta^{f_0,f_1}_{W,F}$ can be realised as $\theta^{f_0,f_1}_{W',F'}$ for some $F' : W' \rightarrow B$, which is normally bordant to $F$. We also show that $\theta^{f_0,f_1}_{W,F}$ is elementary if and only if $W$ is an h-cobordism, and in this case it induces the same isomorphism $E_q(M_0,f_0) \rightarrow E_q(M_1,f_1)$ as $W$ (see Proposition \ref{prop:hcob-elem}). 

In Section \ref{s:real-free} we prove Theorem \ref{thm:main-free}. First we replace $F_0$ with a $q$-connected normal bordism $F : W \rightarrow B$. Using the assumption that $H_{q-1}(B)$ is free, we construct an elementary two-sided quasi-formation that is equivalent to $\theta^{f_0,f_1}_{W,F}$ and induces the given isomorphism $I : E_q(M_0,f_0) \rightarrow E_q(M_1,f_1)$. The equivalence relies on the identity of Theorem \ref{thm:jacobi} and the fact that the relevant extended L-group is trivial when $H_q(B)$ is free. 

When we remove the assumption that $H_{q-1}(B)$ is free, it is no longer possible to find an elementary two-sided quasi-formation that is equivalent to the given obstruction $\theta^{f_0,f_1}_{W,F}$ and induces an arbitrary isomorphism $I$. Thus our goal in general is to determine the set of isomorphisms $I$ that can be realised by a diffeomorphism over $B$. Using the notation set up in Section \ref{ss:sys}, it is necessary that $I$ extends to an isomorphism $\II : \EE_q(M_0, f_0) \rightarrow \EE_q(M_1, f_1)$ between the systems of Q-forms. So we need to prove the other direction of Theorem \ref{thm:main-torsion}, that if $I$ has an extension, then it is realisable. 

We start with the special case when $I$ is an automorphism of $E_q(M,f)$ for some $f : M \rightarrow B$. The first main step is translating the existence of $\II$ into a condition that is expressed directly in terms of $I$. This is achieved in Section \ref{ss:int}, where we show that the condition of Theorem \ref{thm:b-c-intro} is necessary.

The second step is constructing a suitable two-sided quasi-formation that induces $I$, for the allowed automorphisms $I$. Note that for every $I$ there exists a two-sided quasi-formation that induces it (we can take a two-sided quasi-formation that induces $\id_{E_q(M,f)}$, and post-compose the morphism $p_1$ with $I$), the difficulty is in finding one that is the obstruction of a normal bordism. We solve this in Theorem \ref{thm:real-a-b}, where we start with the obstruction of the trivial normal bordism, which is a two-sided quasi-formation that induces $\id_{E_q(M,f)}$, and construct one that induces $I$ by only changing the T-lagrangian. This relies on the conditions that $(b \otimes \id_{H_q(B)})(\hat{\mu}) = 0$ and $c = \id_{\tp}$ for $b$ and $c$ in the matrix form of $I$. Then another argument using Theorem \ref{thm:jacobi} and the vanishing of the extended L-group shows that there is an h-cobordism inducing $I$.

In Section \ref{s:real-tor} we prove Theorem \ref{thm:main-torsion} by reducing it to the case of automorphisms. By \cite[Theorem 1.2]{csn-qfc26} there is an h-cobordism over $B$ between $M_0$ and $M_1$, which induces an isomorphism $E_q(M_0, f_0) \rightarrow E_q(M_1, f_1)$ that extends to an isomorphism between the systems of Q-forms. It differs from $I$ by an automorphism of $E_q(M_0, f_0)$, which also extends, and by the earlier results can be realised by an h-cobordism.

Finally, in Section \ref{s:ci} we prove Theorem \ref{thm:main-ci}. The proof is based on the observation that if an automorphism of $H^*(X_q(\dd))$ fixes $H^2(X_q(\dd)) \cong \Z$, then its restriction to $H^q(X_q(\dd))$ is an automorphism of the Q-form $E_q(X_q(\dd),i)$ of a certain normal $(q{-}1)$-smoothing $i : X_q(\dd) \rightarrow {\C}P^{q+k}$, and hence can be realised by a diffeomorphism by Theorem \ref{thm:main-free}.

\subsection*{Acknowledgements}

I am grateful to Diarmuid Crowley for discussions on extended surgery theory, and to Mark Powell for helpful suggestions on the manuscript.

\section{Two-sided $\ell$-sets} \label{s:l-set}

This section sets up the algebraic background required for the two-sided extended surgery obstruction. We introduce two-sided quasi-formations and two-sided $\ell$-sets, and establish their main properties. In particular, a two-sided quasi-formation is a quasi-formation (see \cite[Section 4.1]{csn-qfc26}) together with some extra data depending on a pair of extended quadratic forms $\UP_0$, $\UP_1$. Its key feature is that if a two-sided quasi-formation is elementary and proper, then it determines an isomorphism $\UP_0 \rightarrow \UP_1$, see Section \ref{ss:proper}. The identity of \cite[Theorem 4.26]{csn-qfc26} has a two-sided analogue, Theorem \ref{thm:jacobi}. As a result, the L-groups can also be defined in this setting, but they are isomorphic to the extended L-groups of \cite[Section 4.3]{csn-qfc26} via the forgetful map.

\subsection{Extended quadratic forms}

This paper builds on the results of \cite{csn-qfc26}, with which the reader is assumed to be familiar. Nevertheless, for convenience, we recall from \cite[Sections 2 and 5]{csn-qfc26} the definition of extended quadratic forms and some related notions, which will be used frequently.

\begin{defin}
An \emph{extended quadratic form over an abelian group $Q$} is a triple $\UM = (M, \lambda, \mu)$, where $M$ is an abelian group, $\lambda : M \times M \rightarrow \Z$ is a symmetric bilinear function, and $\mu : M \rightarrow Q$ is a homomorphism. It is \emph{geometric} with respect to a homomorphism $v : Q \rightarrow \Z_2$, if for every $x \in M$ we have $\varrho_2 \circ \lambda(x,x) = v \circ \mu(x)$, where $\varrho_2 : \Z \rightarrow \Z_2$ is reduction mod $2$. To simplify the terminology, we will also say in this case that $\UM$ is an \emph{extended quadratic form over $(Q,v)$}.
\end{defin}

\begin{rem}
By \cite[Theorem 1.3]{dc-csn-eqf} the pair $(Q,v)$ corresponds to a (symmetric) quadratic form parameter $P$, and by \cite[Lemma 2.29 and Remark 2.30]{dc-csn-eqf} an extended quadratic form over $(Q,v)$ is equivalent to an extended quadratic form over $P$.
\end{rem}

An example of an extended quadratic form is the Q-form of a map.  

\begin{defin}
Let $f : M \rightarrow B$ be a map from a closed oriented $2q$-manifold $M$ (with $q \geq 2$ even) to a topological space $B$. The \emph{Q-form} of $f$ is $E_q(M,f) = (H_q(M), \lambda, \mu)$, where $\lambda : H_q(M) \times H_q(M) \rightarrow \Z$ is the intersection form of $M$ and $\mu = H_q(f) : H_q(M) \rightarrow H_q(B)$. 
\end{defin}

The Q-form $E_q(M,f)$ is an extended quadratic form over $H_q(B)$. If $f$ is a normal map, ie.\ it is covered by some bundle map $\nu_M \rightarrow \xi$ from the stable normal bundle of $M$ to some stable bundle $\xi$ over $B$, then $E_q(M,f)$ is geometric with respect to the homomorphism $\hat{v}_q(\xi) : H_q(B) \rightarrow \Z_2$ determined by the Wu class $v_q(\xi) \in H^q(B;\Z_2)$ of $\xi$.

\begin{defin}
An extended quadratic form $\UM$ over $Q$ is \emph{metabolic}, if it is nonsingular ($\ad \lambda : M \rightarrow \Hom(M, \Z) = M^*$ induces an isomorphism $M / \Tor M \cong M^*$) and contains a T-lagrangian $L$ (a half-rank direct summand in $M$ with $\Tor M \leq L$, $\lambda \big| _{L \times L} = 0$ and $\mu \big| _L = 0$).
\end{defin}

A special case of metabolic extended quadratic forms is a \emph{hyperbolic} form, which is defined by any of the equivalent conditions in the following lemma (see \cite[Lemma 2.13]{csn-qfc26}, and note that a geometric form with $\mu = 0$ is even).

\begin{lem} \label{lem:hyp}
Let $\UM$ be an extended quadratic form over $(Q,v)$. The following are equivalent: 
\begin{compactenum}
\item \label{el:he1} $\UM$ is free, nonsingular and $M = L \oplus L'$ for some lagrangians $L, L' \leq M$.
\item \label{el:he2} $\UM$ is free, metabolic and $\mu = 0$.
\item \label{el:he3} $\UM$ is isomorphic to $\UH_{2k} = (\Z^{2k}, \bigl[ \begin{smallmatrix} 0 & I_k \\ I_k & 0 \end{smallmatrix} \bigr], 0)$, where $2k = \rk M$. 
\end{compactenum}
\end{lem}

\subsection{Two-sided quasi-formations}

For the rest of of Section \ref{s:l-set}, we will fix the data $(Q,v,\UP_0,\UP_1)$ consisting of an abelian group $Q$, a homomorphism $v : Q \rightarrow \Z_2$, and extended quadratic forms $\UP_i = (P_i, \lambda_{P_i}, \mu_{P_i})$ over $(Q,v)$ for $i=0,1$.

\begin{defin} \label{def:2sqf}
A \emph{two-sided quasi-formation} over $(Q,v,\UP_0,\UP_1)$ is a $6$-tuple
\[
(\UM; L, V_0, V_1, p_0, p_1)
\]
where
\begin{compactitem}
\item $\UM = (M, \lambda, \mu)$ is a metabolic extended quadratic form over $(Q,v)$
\item $L \leq M$ is a T-lagrangian
\item $V_0, V_1 \leq M$ are free half-rank direct summands such that $V_i^{\perp} = V_{1-i} \oplus \Tor M$
\item $p_0 : \UV_0 \rightarrow \UP_0$ and $p_1 : \UV_1 \rightarrow -\UPs_1$ are morphisms of extended quadratic forms, where $\UV_i = (V_i, \lambda \big| _{V_i \times V_i}, \mu \big| _{V_i})$
\end{compactitem}
\end{defin}

\begin{defin}
A \emph{two-sided formation} is a two-sided quasi-formation $(\UM; L, V_0, V_1, p_0, p_1)$ such that
\begin{compactitem}
\item $\UM$ is free
\item $V_0 = V_1$ is a lagrangian
\item $p_0 = 0$ and $p_1 = 0$
\end{compactitem}
Hence it is of the form $(\UM; L, K, K, 0, 0)$ for some lagrangians $L$, $K$.
\end{defin}

\begin{defin}
A two-sided quasi-formation $(\UM; L, V_0, V_1, p_0, p_1)$ is called \emph{elementary}, if $M = L \oplus V_0$ (as an internal direct sum).
\end{defin}

\begin{defin}
The \emph{standard rank-$2k$ elementary hyperbolic two-sided formation} is $\HH_{2k} = (\UH_{2k}; \{ 0 \} \times \Z^k, \Z^k \times \{ 0 \}, \Z^k \times \{ 0 \}, 0, 0)$.
\end{defin}

\begin{defin}
Let $(\UM; L, V_0, V_1, p_0, p_1)$ and $(\UMp; L', V'_0, V'_1, p'_0, p'_1)$ be two-sided quasi-formations. They are \emph{isomorphic} if there is an isomorphism $h : \UM \rightarrow \UMp$ such that $L' = h(L)$, $V_i' = h(V_i)$ and $p_i = p'_i \circ h : V_i \rightarrow P_i$. 
\end{defin}

\begin{defin}
Let $(\UM; L, V_0, V_1, p_0, p_1)$ be a two-sided quasi-formation and $(\UMp; L', K', K', 0, 0)$ be a two-sided formation. Their \emph{direct sum} is $(\UM; L, V_0, V_1, p_0, p_1) \oplus (\UMp; L', K', K', 0, 0) = (\UM \oplus \UMp; L \oplus L', V_0 \oplus K', V_1 \oplus K', p_0 \oplus 0, p_1 \oplus 0)$. 
\end{defin}

\begin{rem}
The direct sum of two two-sided quasi-formations is not defined, because if $p : \UV \rightarrow \UP$ and $p' : \UVp \rightarrow \UP$ are morphisms of extended quadratic forms, then the map $p \oplus p' : \UV \oplus \UVp \rightarrow \UP$ is not a morphism in general.
\end{rem}

\begin{defin}
Let $(\UM; L, V_0, V_1, p_0, p_1)$ and $(\UMp; L', V'_0, V'_1, p'_0, p'_1)$ be two-sided quasi-formations. They are \emph{stably isomorphic} if there are integers $k,l \geq 0$ such that  $(\UM; L, V_0, V_1, p_0, p_1) \oplus \HH_{2k} \cong (\UMp; L', V'_0, V'_1, p'_0, p'_1) \oplus \HH_{2l}$. Stable isomorphism will be denoted by ${} \si {}$. 
\end{defin}

\begin{defin} \label{def:gc}
Given $(Q,v,\UP_0,\UP_1)$, a torsion abelian group $R$ and an even integer $q \geq 0$, let
\[
\gc^T_{2q+1}(Q,v,\UP_0,\UP_1,R)
\]
denote the class of two-sided quasi-formations $(\UM; L, V_0, V_1, p_0, p_1)$ such that $\UM = (M, \lambda, \mu)$ is full (ie.\ $\mu$ is surjective) and $\Tor M \cong R$. Further, let
\[
\begin{aligned}
\gs^T_{2q+1}(Q,v,\UP_0,\UP_1,R) &= \gc^T_{2q+1}(Q,v,\UP_0,\UP_1,R) / {\cong} \\
\s^T_{2q+1}(Q,v,\UP_0,\UP_1,R) &= \gc^T_{2q+1}(Q,v,\UP_0,\UP_1,R) / {\si} \\
\ell^T_{2q+1}(Q,v,\UP_0,\UP_1,R) &= \gc^T_{2q+1}(Q,v,\UP_0,\UP_1,R) / {\sim}
\end{aligned}
\]
denote the set of isomorphism classes, stable isomorphism classes, and equivalence classes of such two-sided quasi-formations, respectively. The equivalence relation $\sim$ is generated by the following elementary equivalences: 
\begin{compactitem}
\item $(\UM; L, V_0, V_1, p_0, p_1) \sim (\UMp; L', V'_0, V'_1, p'_0, p'_1)$ if $(\UM; L, V_0, V_1, p_0, p_1) \si (\UMp; L', V'_0, V'_1, p'_0, p'_1)$
\item $(\UM \oplus \UH_2; L \oplus (\{ 0 \} \times \Z), V_0, V_1, p_0, p_1) \sim (\UM \oplus \UH_2; L \oplus (\Z \times \{ 0 \}), V_0, V_1, p_0, p_1)$, where $L$ is a T-lagrangian in $\UM$ and $V_i$ is a free half-rank direct summand in $\UM \oplus \UH_2$
\end{compactitem}
The equivalence class of $(\UM; L, V_0, V_1, p_0, p_1)$ will be denoted by $[\UM; L, V_0, V_1, p_0, p_1]$. 
\end{defin}

Note that the class $\gc^T_{2q+1}(Q,v,\UP_0,\UP_1,R)$ does not depend on the choice of $q$. 

Recall the definition of $\gc^T_{2q+1}(Q,v,R)$ and its quotients from \cite[Section 4.1]{csn-qfc26}.

\begin{defin}
The natural forgetful maps 
\[
\begin{aligned}
\gc^T_{2q+1}(Q,v,\UP_0,\UP_1,R) &\rightarrow \gc^T_{2q+1}(Q,v,R) \\
\gs^T_{2q+1}(Q,v,\UP_0,\UP_1,R) &\rightarrow \gs^T_{2q+1}(Q,v,R) \\
\s^T_{2q+1}(Q,v,\UP_0,\UP_1,R) &\rightarrow \s^T_{2q+1}(Q,v,R) \\
\ell^T_{2q+1}(Q,v,\UP_0,\UP_1,R) &\rightarrow \ell^T_{2q+1}(Q,v,R)
\end{aligned}
\] 
are given by $(\UM; L, V_0, V_1, p_0, p_1) \mapsto (\UM; L, V_0)$.
\end{defin}

\begin{defin}
Given $(Q,v,\UP_0,\UP_1)$ and an even integer $q \geq 0$, let
\[
\begin{aligned}
\gc_{2q+1}(Q,v,\UP_0,\UP_1) &= \gc^T_{2q+1}(Q,v,\UP_0,\UP_1,0) \\
\gs_{2q+1}(Q,v,\UP_0,\UP_1) &= \gs^T_{2q+1}(Q,v,\UP_0,\UP_1,0) \\
\s_{2q+1}(Q,v,\UP_0,\UP_1) &= \s^T_{2q+1}(Q,v,\UP_0,\UP_1,0) \\
\ell_{2q+1}(Q,v,\UP_0,\UP_1) &= \ell^T_{2q+1}(Q,v,\UP_0,\UP_1,0) 
\end{aligned}
\]
\end{defin}

\begin{defin} \label{def:red}
Suppose that $\UM = (M, \lambda, \mu)$ is an extended quadratic form over $(Q,v)$ such that $\mu \big| _{\Tor M} = 0$. We define the \emph{free reduction} of $\UM$, denoted
\[
\UbM = (\bar{M}, \bar{\lambda}, \bar{\mu})
\]
where $\bar{M} = M / \Tor M$, and $\bar{\lambda} : \bar{M} \times \bar{M} \rightarrow \Z$ and $\bar{\mu} : \bar{M} \rightarrow Q$ are the maps induced by $\lambda$ and $\mu$. 

If $X \leq M$ is a subgroup, then $\bar{X} = \pi_M(X) \leq \bar{M}$ will denote its image in $\bar{M}$, where $\pi_M : M \rightarrow M / \Tor M$ is the quotient map.

Suppose that $x = (\UM; L, V_0, V_1, p_0, p_1)$ is a two-sided quasi-formation over $(Q,v,\UP_0,\UP_1)$ (in particular, $\mu \big| _{\Tor M} = 0$, since $L$ is a T-lagrangian). We define the \emph{free reduction} of $x$, denoted
\[
\bar{x} = (\UbM; \bar{L}, \bar{V}_0, \bar{V}_1, \bar{p}_0, \bar{p}_1)
\]
where $\bar{p}_i = p_i \circ (\pi_M \big| _{V_i})^{-1} : \bar{V}_i \rightarrow P_i$ (using that $\pi_M \big| _{V_i} : V_i \rightarrow \bar{V}_i$ is an isomorphism).
\end{defin}

\begin{prop} \label{prop:red-surj}
a) Free reduction induces a natural surjective map $\ell^T_{2q+1}(Q,v,\UP_0,\UP_1,R) \rightarrow \ell_{2q+1}(Q,v,\UP_0,\UP_1)$.

b) If $x \in \gc^T_{2q+1}(Q,v,\UP_0,\UP_1,R)$ and $\hat{y} \in \gc_{2q+1}(Q,v,\UP_0,\UP_1)$ such that $\bar{x} \sim \hat{y}$, then there is a $y \in \gc^T_{2q+1}(Q,v,\UP_0,\UP_1,R)$ such that $x \sim y$ and $\bar{y} \si \hat{y}$.
\end{prop}

\begin{proof}
a) If $(\UM; L, V_0, V_1, p_0, p_1) \in \gc_{2q+1}(Q,v,\UP_0,\UP_1)$, then it is isomorphic to the free reduction of $(\UM; L, V_0, V_1, p_0, p_1) \oplus (\UR; R, 0, 0, 0, 0) \in \gc_{2q+1}(Q,v,\UP_0,\UP_1,R)$, where $\UR = (R, 0, 0)$. Hence free reduction defines a surjective map $\gs^T_{2q+1}(Q,v,\UP_0,\UP_1,R) \rightarrow \gs_{2q+1}(Q,v,\UP_0,\UP_1)$. For any $x \in \gc_{2q+1}(Q,v,\UP_0,\UP_1,R)$ we have $\overline{x \oplus \HH_{2k}} \cong \bar{x} \oplus \HH_{2k}$, showing that there is a well-defined induced map $\s^T_{2q+1}(Q,v,\UP_0,\UP_1,R) \rightarrow \s_{2q+1}(Q,v,\UP_0,\UP_1)$. Finally, if $x_j = (\UM \oplus \UH_2; L \oplus K_j, V_0, V_1, p_0, p_1) \in \gc_{2q+1}(Q,v,\UP_0,\UP_1,R)$ for $K_1 = \{ 0 \} \times \Z$ and $K_2 = \Z \times \{ 0 \}$, then $\bar{x}_j \cong (\UbM \oplus \UH_2; \bar{L} \oplus K_j, \bar{V}_0, \bar{V}_1, \bar{p}_0, \bar{p}_1)$, so $\bar{x}_1 \sim \bar{x}_2$. Therefore the free reduction is also well-defined on equivalence classes. 

b) It is enough to check this when $\bar{x}$ and $\hat{y}$ are related by a single elementary equivalence, which can be assumed to be of the second type (otherwise, if $\bar{x} \si \hat{y}$, then we can take $y=x$). In this case $x$, $\bar{x}$ and $\hat{y}$ are of the form $x = (\UM; L, V_0, V_1, p_0, p_1)$, $\bar{x} = (\UbM; \bar{L}, \bar{V}_0, \bar{V}_1, \bar{p}_0, \bar{p}_1) \cong (\UN \oplus \UH_2; L' \oplus (\{ 0 \} \times \Z), V'_0, V'_1, p'_0, p'_1)$ and $\hat{y} \cong (\UN \oplus \UH_2; L' \oplus (\Z \times \{ 0 \}), V'_0, V'_1, p'_0, p'_1)$. So there are isomorphisms $\UM \cong \UbM \oplus \UR \cong \UN \oplus \UH_2 \oplus \UR$ restricting to $L \cong \bar{L} \oplus R \cong L' \oplus (\{ 0 \} \times \Z) \oplus R$. Let $K$ be the T-lagrangian in $\UM$ corresponding to $L' \oplus (\Z \times \{ 0 \}) \oplus R$, and take $y = (\UM; K, V_0, V_1, p_0, p_1)$. Then $x \sim y$ and $\bar{y} = (\UbM; \bar{K}, \bar{V}_0, \bar{V}_1, \bar{p}_0, \bar{p}_1) \cong \hat{y}$.
\end{proof}

\subsection{Elementary proper two-sided quasi-formations} \label{ss:proper}

Now we describe the conditions under which a two-sided quasi-formation over $(Q,v,\UP_0,\UP_1)$ determines an isomorphism $\UP_0 \xra{\cong} \UP_1$. 

\begin{defin} \label{def:proper}
A two-sided quasi-formation $(\UM; L, V_0, V_1, p_0, p_1)$ is \emph{proper}, if $p_i : V_i \rightarrow P_i$ is surjective for both $i$ and $\Ker p_0 = \Ker p_1 \leq V_0 \cap V_1$. 
\end{defin}

Note that $(\UM; L, V_0, V_1, p_0, p_1)$ is proper if and only if $(\UM; L, V_0, V_1, p_0, p_1) \oplus \HH_{2k}$ is proper, for any $k$. Moreover, since being proper is independent of the choice of T-lagrangian $L$, we see that it is an $\sim$-invariant property.

Recall that a two-sided quasi-formation $(\UM; L, V_0, V_1, p_0, p_1)$ is called elementary if $M = L \oplus V_0$. As in the case of quasi-formations, this property is invariant under stable isomorphism.

\begin{defin}
An element of $\gs^T_{2q+1}(Q,v,\UP_0,\UP_1,R)$ or $\s^T_{2q+1}(Q,v,\UP_0,\UP_1,R)$ is called \emph{elementary}, if its representatives are elementary. An element of $\ell^T_{2q+1}(Q,v,\UP_0,\UP_1,R)$ is called \emph{elementary}, if it has an elementary representative. 
\end{defin}

\begin{lem} \label{lem:pr-el-equiv} 
Suppose that $x = (\UM; L, V_0, V_1, p_0, p_1) \in \gs^T_{2q+1}(Q,v,\UP_0,\UP_1,R)$ is a two-sided quasi-formation and $\bar{x} = (\UbM; \bar{L}, \bar{V}_0, \bar{V}_1, \bar{p}_0, \bar{p}_1) \in \gs_{2q+1}(Q,v,\UP_0,\UP_1)$ is its free reduction. Then 

a) If $x$ is proper, then $\bar{x}$ is proper.

b) $x$ is elementary if and only if $\bar{x}$ is elementary. 

c) $M = L \oplus V_0$ (that is, $x$ is elementary) if and only if $M = L \oplus V_1$.

d) If $x$ is elementary, then projection along $L$ defines an isomorphism $V_0 \rightarrow V_1$.
\end{lem}

\begin{proof}
a) This follows from the fact that $\pi_M$ restricts to an isomorphism $V_i \cong \bar{V}_i$ for both $i$, and to an injective map $V_0 \cap V_1 \rightarrow \bar{V}_0 \cap \bar{V}_1$. In particular, $\Image p_i = \Image \bar{p}_i $ and $\Ker p_i$ is identified with $\Ker \bar{p}_i$.

b) Note that $x$ is elementary if and only if the composition $V_0 \rightarrow M \rightarrow M/L$ is an isomorphism. Since $V_i$ is free and $\Tor M \leq L$, the vertical maps induced by $\pi_M$ are isomorphisms in the diagram 
\[
\xymatrix{
V_i \ar[r] \ar[d]_{\cong} & M/L \ar[d]^{\cong} \\
\bar{V}_i \ar[r] & \bar{M}/\bar{L} 
}
\]
hence $V_i \rightarrow M/L$ is an isomorphism if and only if $\bar{V}_i \rightarrow \bar{M}/\bar{L}$ is an isomorphism.

c) By part b) it is enough to check this for the free reduction. By \cite[Lemma 2.7]{csn-qfc26}, if $\bar{M} = \bar{L} \oplus \bar{V}_i$, then  $\bar{M} = \bar{L}^{\perp} \oplus \bar{V}_i^{\perp} = \bar{L} \oplus \bar{V}_{1-i}$.

d) By part c), if $x$ is elementary, then $V_0 \cong M/L \cong V_1$.
\end{proof}

\begin{defin} \label{def:ix}
Suppose that $x = (\UM; L, V_0, V_1, p_0, p_1)$ is an elementary proper two-sided quasi-formation. Then we define the isomorphism $I_x : P_0 \rightarrow P_1$ as the composition
\[
P_0 \xra{\cong} V_0 / \Ker p_0 \xra{\cong} V_1 / \Ker p_1 \xra{\cong} P_1
\]
In the first and third map we use the isomorphism $V_i / \Ker p_i \rightarrow P_i$ induced by $p_i$ (and the fact that $p_i$ is surjective). For the second map, we use that by Lemma \ref{lem:pr-el-equiv} d) projection along $L$ gives an isomorphism $V_0 \rightarrow V_1$. Since $\Ker p_0 = \Ker p_1 \leq V_0 \cap V_1$ and the projection is the identity on $V_0 \cap V_1$, there is an induced isomorphism $V_0 / \Ker p_0 \rightarrow V_1 / \Ker p_1$.
\end{defin}

\begin{prop}
Suppose that $x = (\UM; L, V_0, V_1, p_0, p_1)$ is an elementary proper two-sided quasi-formation. Then $I_x : \UP_0 \rightarrow \UP_1$ is an isomorphism of extended quadratic forms. 
\end{prop}

\begin{proof}
Since $p_0 : \UV_0 \rightarrow \UP_0$ is a morphism of extended quadratic forms, $\lambda \big| _{V_0 \times \Ker p_0} = 0$, and $\mu \big| _{\Ker p_0} = 0$. So there is an induced form $\UhV_0$ on $\hat{V}_0 = V_0 / \Ker p_0$, and $p_0$ induces an isomorphism $\UhV_0 \rightarrow \UP_0$ of extended quadratic forms. Similarly there is an induced form $\UhV_1$ on $\hat{V}_1 = V_1 / \Ker p_1$, and $p_1$ induces an isomorphism $\UhV_1 \rightarrow -\UPs_1$ of extended quadratic forms, which is also an isomorphism $-\UhVs_1 \rightarrow \UP_1$. It remains to show that the projection along $L$ induces an isomorphism $\UhV_0 \rightarrow -\UhVs_1$. 

Suppose that $x_1, x_2 \in V_0$, and their images under the projection along $L$ are $y_1, y_2 \in V_1$, that is, $y_1 = x_1 + l_1$ and $y_2 = x_2 + l_2$ for some $l_1, l_2 \in L$. Since $V_{1-i} \leq V_i^{\perp}$, we have $\lambda(x_1,y_2) = \lambda(y_1,x_2) = 0$. Thus $\lambda(y_1,y_2) = \lambda(x_1+l_1,x_2+l_2) = \lambda(x_1,x_2+l_2) + \lambda(x_1+l_1,x_2) + \lambda(l_1,l_2) - \lambda(x_1,x_2) = - \lambda(x_1,x_2)$. We also have $\mu(y_i) = \mu(x_i) + \mu(l_i) = \mu(x_i)$. This means that the projection along $L$ is an isomorphism $\UV_0 \rightarrow -\UVs_1$ of extended quadratic forms, therefore it induces an isomorphism $\UhV_0 \rightarrow -\UhVs_1$. 
\end{proof}

\begin{lem} \label{lem:I-stab-free} 
Suppose that $x = (\UM; L, V_0, V_1, p_0, p_1)$ is an elementary proper two-sided quasi-formation. Then, by our earlier observations, $x \oplus \HH_{2k}$ (for any $k \geq 0$) and $\bar{x}$ are also elementary and proper, and we have

a) $I_x = I_{x \oplus \HH_{2k}}$.

b) $I_x = I_{\bar{x}}$. 
\end{lem}

\begin{proof}
The compositions defining $I_x$, $I_{x \oplus \HH_{2k}}$ and $I_{\bar{x}}$ fit into commutative diagrams
\[
\xymatrix{
P_0 \ar[r]^-{\cong} \ar@{=}[d] & V_0 / \Ker p_0 \ar[r]^-{\cong} \ar[d]_-{\cong} & V_1 / \Ker p_1 \ar[r]^-{\cong} \ar[d]^-{\cong} & P_1 \ar@{=}[d] \\
P_0 \ar[r]^-{\cong} & (V_0 \oplus (\Z^k \times \{ 0 \})) / \Ker (p_0 \oplus 0) \ar[r]^-{\cong} & (V_1 \oplus (\Z^k \times \{ 0 \})) / \Ker (p_1 \oplus 0) \ar[r]^-{\cong} & P_1 \\
}
\]
and
\[
\xymatrix{
P_0 \ar[r]^-{\cong} \ar@{=}[d] & V_0 / \Ker p_0 \ar[r]^-{\cong} \ar[d]_-{\cong} & V_1 / \Ker p_1 \ar[r]^-{\cong} \ar[d]^-{\cong} & P_1 \ar@{=}[d] \\
P_0 \ar[r]^-{\cong} & \bar{V}_0 / \Ker \bar{p}_0 \ar[r]^-{\cong} & \bar{V}_1 / \Ker \bar{p}_1 \ar[r]^-{\cong} & P_1 \\
}
\]
induced by the restrictions of the inclusion $M \rightarrow M \oplus \Z^{2k}$ and of $\pi_M$, respectively.
\end{proof}

In particular, it makes sense for a stable isomorphism class $x \in \s^T_{2q+1}(Q,v,\UP_0,\UP_1,R)$ to be elementary or proper, and if it is both, then it induces a well-defined isomorphism $I_x : \UP_0 \rightarrow \UP_1$.

\subsection{The set $\ell_{2q+1}(Q,v,\UP_0,\UP_1)$}

Now we establish the basic properties of $\ell_{2q+1}(Q,v,\UP_0,\UP_1)$, analogously to the results in \cite[Section 4.2]{csn-qfc26}. Most of the proofs can be adapted to the new setting in a straightforward way, replacing quasi-formations $(\UM; L, V)$ with two-sided quasi-formations $(\UM; L, V_0, V_1, p_0, p_1)$, and for any operation on $V = V_0$, applying the corresponding operations to $V_1$, $p_0$ and $p_1$. To illustrate this, we will give a full proof of the first lemma, and of the main Theorem \ref{thm:jacobi}; for the remaining statements we refer to \cite{csn-qfc26}.

Recall the definition of the subgroup $\RU(\UM,L) \leq \Aut(\UM)$ from \cite[Section 3]{csn-qfc26}.

\begin{lem} \label{lem:ru-equiv}
Suppose that $(\UM; L, V_0, V_1, p_0, p_1) \in \gc_{2q+1}(Q,v,\UP_0,\UP_1)$. If $\Psi \in \RU(\UM,L)$, then we have $(\UM; L, V_0, V_1, p_0, p_1) \sim (\UM; L, \Psi(V_0), \Psi(V_1), p_0 \circ \Psi^{-1}, p_1 \circ \Psi^{-1})$.
\end{lem}

\begin{proof}
It is enough to prove the statement when $\Psi$ is a generator of $\RU(\UM,L)$. 

If $\Psi(L) = L$, then $\Psi$ is an isomorphism between $(\UM; L, V_0, V_1, p_0, p_1)$ and $(\UM; L, \Psi(V_0), \Psi(V_1), p_0 \circ \Psi^{-1}, p_1 \circ \Psi^{-1})$.

If there is an isomorphism $I : \UM \rightarrow \UH_2 \oplus \UMp$ such that $I(L) = ( \{0 \} \times \Z) \oplus L'$ for some lagrangian $L'$ in $\UMp$ and $\Psi = I^{-1} \circ (\sigma \oplus \id_{M'}) \circ I$, then $I \circ \Psi(L) = (\sigma \oplus \id_{M'}) \circ I(L) = (\sigma \oplus \id_{M'})(( \{0 \} \times \Z) \oplus L') = (\Z \times \{0 \} ) \oplus L'$. So 
\[
\begin{aligned}
 &(\UM; L, V_0, V_1, p_0, p_1) \cong \\
 &\cong (\UH_2 \oplus \UMp; I \circ \Psi(L), I \circ \Psi(V_0), I \circ \Psi(V_1), p_0 \circ \Psi^{-1} \circ I^{-1}, p_1 \circ \Psi^{-1} \circ I^{-1}) = \\
 &= (\UH_2 \oplus \UMp; (\Z \times \{0 \} ) \oplus L', I \circ \Psi(V_0), I \circ \Psi(V_1), p_0 \circ \Psi^{-1} \circ I^{-1}, p_1 \circ \Psi^{-1} \circ I^{-1}) \sim \\
 &\sim (\UH_2 \oplus \UMp; ( \{0 \} \times \Z) \oplus L', I \circ \Psi(V_0), I \circ \Psi(V_1), p_0 \circ \Psi^{-1} \circ I^{-1}, p_1 \circ \Psi^{-1} \circ I^{-1}) = \\
 &= (\UH_2 \oplus \UMp; I(L), I \circ \Psi(V_0), I \circ \Psi(V_1), p_0 \circ \Psi^{-1} \circ I^{-1}, p_1 \circ \Psi^{-1} \circ I^{-1}) \cong \\
 &\cong (\UM; L, \Psi(V_0), \Psi(V_1), p_0 \circ \Psi^{-1}, p_1 \circ \Psi^{-1})
\end{aligned}
\]
where we used the isomorphisms $I \circ \Psi$ and $I$. Therefore the statement holds for every generator of $\RU(\UM,L)$.
\end{proof}

\begin{lem} \label{lem:hyp-zero}
If $(\UM; L, V_0, V_1, p_0, p_1) \in \gc_{2q+1}(Q,v,\UP_0,\UP_1)$ and $k \geq 0$, then $(\UM; L, V_0, V_1, p_0, p_1) \sim (\UM; L, V_0, V_1, p_0, p_1) \oplus (\UH_{2k}; \{ 0 \} \times \Z^k, \{ 0 \} \times \Z^k, \{ 0 \} \times \Z^k, 0, 0)$. 
\end{lem}

\begin{proof}
Follows the proof of \cite[Lemma 4.20]{csn-qfc26}.
\end{proof}

\begin{prop} \label{prop:lmonoid-zero}
a) The set $\ell_{2q+1}(Q,v,\UP_0,\UP_1)$ is non-empty. 

b) All elements of $\gc_{2q+1}(Q,v,\UP_0,\UP_1)$ of the form $(\UM; L, L, L, 0, 0)$ are equivalent.

c) If $(\UM; L, L, L, 0, 0), x \in \gc_{2q+1}(Q,v,\UP_0,\UP_1)$, then $(\UM; L, L, L, 0, 0) \oplus x \sim x$.
\end{prop}

\begin{proof}
Follows the proof of \cite[Proposition 4.21]{csn-qfc26}. For part c), the calculation that showed that $I(V \oplus L) =  V \oplus (\{ 0 \} \times \Z^k)$ can be applied to both $V_0$ and $V_1$, hence $I(V_i \oplus L) =  V_i \oplus (\{ 0 \} \times \Z^k)$ for both $i$. Further, one checks that $(p_i \oplus 0) \circ I = p_i \oplus 0 : V_i \oplus L \rightarrow P_i$.
\end{proof}

\begin{lem} \label{lem:nkk-zero}
Suppose that $(\UM; L, V_0, V_1, p_0, p_1) \in \gc_{2q+1}(Q,v,\UP_0,\UP_1)$, $\UN$ is a free metabolic form over $(Q,v)$ and $K$ is a lagrangian in $\UN$ (so that $(\UM; L, V_0, V_1, p_0, p_1) \oplus (\UN; K, K, K, 0, 0) \in \gc_{2q+1}(Q,v,\UP_0,\UP_1)$). Then $(\UM; L, V_0, V_1, p_0, p_1) \sim (\UM; L, V_0, V_1, p_0, p_1) \oplus (\UN; K, K, K, 0, 0)$. 
\end{lem}

\begin{proof}
Follows the proof of \cite[Lemma 4.24]{csn-qfc26}.
\end{proof}

\begin{lem} \label{lem:ru-st-equiv}
Suppose that $(\UM; L, V_0, V_1, p_0, p_1) \in \gc_{2q+1}(Q,v,\UP_0,\UP_1)$. If $\Psi \in \RU_{\st}(\UM,L)$, then 
\[
(\UM; L, V_0, V_1, p_0, p_1) \sim (\UM; L, \Psi(V_0), \Psi(V_1), p_0 \circ \Psi^{-1}, p_1 \circ \Psi^{-1}) \text{\,.}
\] 
\end{lem}

\begin{proof}
Follows the proof of \cite[Lemma 4.25]{csn-qfc26}.
\end{proof}

\begin{thm} \label{thm:jacobi}
Suppose that $(\UM; L, V_0, V_1, p_0, p_1) \in \gc_{2q+1}(Q,v,\UP_0,\UP_1)$, and let $K$ be a lagrangian in $\UM$ (so $(\UM; K, L, L, 0, 0), (\UM; K, V_0, V_1, p_0, p_1) \in \gc_{2q+1}(Q,v,\UP_0,\UP_1)$). Then 
\[
[\UM; K, L, L, 0, 0] \oplus [\UM; L, V_0, V_1, p_0, p_1] = [\UM; K, V_0, V_1, p_0, p_1] \in \ell_{2q+1}(Q,v,\UP_0,\UP_1) \text{\,.}
\]
\end{thm}

\begin{proof}
By \cite[Theorem 2.27 a)]{csn-qfc26} there is an isomorphism $\Phi : \UM \oplus \UH_{2k} \rightarrow \UM \oplus \UH_{2k}$ such that $\Phi(L \oplus (\{ 0 \} \times \Z^k)) = K \oplus (\{ 0 \} \times \Z^k)$ for some $k \geq 0$. Let $\UtM = \UM \oplus \UH_{2k}$, $\tilde{L} = L \oplus (\{ 0 \} \times \Z^k)$, $\tilde{K} = K \oplus (\{ 0 \} \times \Z^k)$, $\tilde{V}_i = V_i \oplus (\Z^k \times \{ 0 \})$ and $\tilde{p}_i = p_i \oplus 0$, then $[\UM; L, V_0, V_1, p_0, p_1] = [\UtM; \tilde{L}, \tilde{V}_0, \tilde{V}_1, \tilde{p}_0, \tilde{p}_1]$, $[\UM; K, V_0, V_1, p_0, p_1] = [\UtM; \tilde{K}, \tilde{V}_0, \tilde{V}_1, \tilde{p}_0, \tilde{p}_1]$ and (by Lemma \ref{lem:hyp-zero}) $[\UM; K, L, L, 0, 0] = [\UtM; \tilde{K}, \tilde{L}, \tilde{L}, 0, 0]$.

By \cite[Theorem 3.7]{csn-qfc26} $\Phi \oplus \Phi^{-1} \in \RU_{\st}(\UtM \oplus \UtM, \tilde{K} \oplus \tilde{K})$. Therefore we have 
\[
\begin{aligned}
 &(\UtM; \tilde{K}, \tilde{L}, \tilde{L}, 0, 0) \oplus (\UtM; \tilde{L}, \tilde{V}_0, \tilde{V}_1, \tilde{p}_0, \tilde{p}_1) \cong \\ 
 &\cong (\UtM; \tilde{K}, \tilde{L}, \tilde{L}, 0, 0) \oplus (\UtM; \tilde{K}, \Phi(\tilde{V}_0), \Phi(\tilde{V}_1), \tilde{p}_0 \circ \Phi^{-1}, \tilde{p}_1 \circ \Phi^{-1}) = \\
 &= (\UtM \oplus \UtM; \tilde{K} \oplus \tilde{K}, \tilde{L} \oplus \Phi(\tilde{V}_0), \tilde{L} \oplus \Phi(\tilde{V}_1), 0 \oplus \tilde{p}_0 \circ \Phi^{-1}, 0 \oplus \tilde{p}_1 \circ \Phi^{-1}) \sim \\
 &\sim (\UtM \oplus \UtM; \tilde{K} \oplus \tilde{K}, \tilde{K} \oplus \tilde{V}_0, \tilde{K} \oplus \tilde{V}_1, 0 \oplus \tilde{p}_0, 0 \oplus \tilde{p}_1) = \\
 &= (\UtM; \tilde{K}, \tilde{K}, \tilde{K}, 0, 0) \oplus (\UtM; \tilde{K}, \tilde{V}_0, \tilde{V}_1, \tilde{p}_0, \tilde{p}_1) \sim \\ 
 &\sim (\UtM; \tilde{K}, \tilde{V}_0, \tilde{V}_1, \tilde{p}_0, \tilde{p}_1) 
\end{aligned}
\]
using the isomorphism $\Phi$ and Lemmas \ref{lem:ru-st-equiv} and \ref{lem:nkk-zero}. 
\end{proof}

\subsection{The group $L_{2q+1}(Q,v,\UP_0,\UP_1)$}

Note that if a two-sided quasi-formation is equivalent to a two-sided formation, then in it is a two-sided formation itself.

\begin{defin}
For $(Q,v,\UP_0,\UP_1)$ and an even integer $q \geq 0$ let
\[
L_{2q+1}(Q,v,\UP_0,\UP_1) \subseteq \ell_{2q+1}(Q,v,\UP_0,\UP_1) 
\]
be the set of equivalence classes of two-sided formations in $\gc_{2q+1}(Q,v,\UP_0,\UP_1)$.
\end{defin}

\begin{prop} \label{prop:L-action}
$L_{2q+1}(Q,v,\UP_0,\UP_1)$ is an abelian group under direct sum. It acts on the set $\ell_{2q+1}(Q,v,\UP_0,\UP_1)$ via direct sum.
\end{prop}

\begin{proof}
By Proposition \ref{prop:lmonoid-zero} the equivalence class of two-sided formations of the form $(\UM; L, L, L, 0, 0)$ is a neutral element for direct sum, and by Theorem \ref{thm:jacobi} every element is invertible, namely the inverse of $[\UM; L, K, K, 0, 0]$ is $[\UM; K, L, L, 0, 0]$. All other required properties follow immediately from the definitions. 
\end{proof}

\begin{prop} \label{prop:L-forget}
a) The forgetful map restricts to an isomorphism $L_{2q+1}(Q,v,\UP_0,\UP_1) \rightarrow L_{2q+1}(Q,v)$. 

b) If $Q$ is free, then $L_{2q+1}(Q,v,\UP_0,\UP_1) \cong 0$.
\end{prop}

\begin{proof}
a) The inverse is given by $[\UM; L, K] \mapsto [\UM; L, K, K, 0, 0]$.

b) By \cite[Theorem 4.33]{csn-qfc26}, $L_{2q+1}(Q,v) \cong 0$.
\end{proof}

\subsection{The two-sided $\ell$-set associated to a pair of normal maps}

Recall from \cite[Definition 4.34]{csn-qfc26} that $\hat{v}_q(\xi) : H_q(B) \rightarrow \Z_2$ denotes the homomorphism determined by the Wu class $v_q(\xi) \in H^q(B;\Z_2)$ of a stable bundle $\xi$ over a space $B$.

\begin{defin} \label{def:l-space}
For a simply-connected space $B$, a stable bundle $\xi$ over it, an even integer $q \geq 2$, closed oriented $2q$-manifolds $M_0$, $M_1$, and normal maps $f_i : M_i \rightarrow B$ let 
\[
\begin{aligned}
\gc_{2q+1}(B,\xi,f_0,f_1) &= \gc_{2q+1}(H_q(B),\hat{v}_q(\xi),E_q(M_0,f_0),E_q(M_1,f_1)) \\
\gs_{2q+1}(B,\xi,f_0,f_1) &= \gs_{2q+1}(H_q(B),\hat{v}_q(\xi),E_q(M_0,f_0),E_q(M_1,f_1)) \\
\s_{2q+1}(B,\xi,f_0,f_1) &= \s_{2q+1}(H_q(B),\hat{v}_q(\xi),E_q(M_0,f_0),E_q(M_1,f_1)) \\
\ell_{2q+1}(B,\xi,f_0,f_1) &= \ell_{2q+1}(H_q(B),\hat{v}_q(\xi),E_q(M_0,f_0),E_q(M_1,f_1)) \\
L_{2q+1}(B,\xi,f_0,f_1) &= L_{2q+1}(H_q(B),\hat{v}_q(\xi),E_q(M_0,f_0),E_q(M_1,f_1)) 
\end{aligned}
\]
and 
\[
\begin{aligned}
\gc^T_{2q+1}(B,\xi,f_0,f_1) &= \gc^T_{2q+1}(H_q(B),\hat{v}_q(\xi),E_q(M_0,f_0),E_q(M_1,f_1),\Tor H_{q-1}(B)) \\
\gs^T_{2q+1}(B,\xi,f_0,f_1) &= \gs^T_{2q+1}(H_q(B),\hat{v}_q(\xi),E_q(M_0,f_0),E_q(M_1,f_1),\Tor H_{q-1}(B)) \\
\s_{2q+1}^T(B,\xi,f_0,f_1) &= \s^T_{2q+1}(H_q(B),\hat{v}_q(\xi),E_q(M_0,f_0),E_q(M_1,f_1),\Tor H_{q-1}(B)) \\
\ell^T_{2q+1}(B,\xi,f_0,f_1) &= \ell^T_{2q+1}(H_q(B),\hat{v}_q(\xi),E_q(M_0,f_0),E_q(M_1,f_1),\Tor H_{q-1}(B)) 
\end{aligned}
\]
\end{defin}

\section{The two-sided extended surgery obstruction $\theta^{f_0,f_1}_{W,F}$} \label{s:obstr-def}

Let $B$ be a simply-connected space, $\xi$ a stable vector bundle over $B$ and $q \geq 2$ an even integer. Let $M_i$ be a closed oriented $2q$-manifold and let $f_i : M_i \rightarrow B$ be a normal $(q{-}1)$-smoothing over $(B,\xi)$, for $i=0,1$. Note that $f_i$ is covered by a bundle map $\nu_{M_i} \rightarrow \xi$, which will not be included in the notation. Let $F : W \rightarrow B$ be a normal bordism between $f_0$ and $f_1$. Assume that $F$ is also a normal $(q{-}1)$-smoothing, and that $\chi(M_0)=\chi(M_1)$. 

We will construct the two-sided extended surgery obstruction $\theta^{f_0,f_1}_{W,F} \in \s^T_{2q+1}(B,\xi,f_0,f_1)$ analogously to $\theta_{W,F}$. Given a CW-decomposition $c$ of $W$ and a regular neighbourhood $U \subset \interior W$ of an embedding $\sk_q W \rightarrow \interior W$, we will define 
\[
\begin{aligned}
\theta^{f_0,f_1}_{U,F} &\in \gc^T_{2q+1}(B,\xi,f_0,f_1) \\
\theta^{f_0,f_1}_{c,F} &\in \gs^T_{2q+1}(B,\xi,f_0,f_1) \\
\theta^{f_0,f_1}_{W,F} &\in \s_{2q+1}^T(B,\xi,f_0,f_1) \\
[\theta^{f_0,f_1}_{W,F}] &\in \ell^T_{2q+1}(B,\xi,f_0,f_1)
\end{aligned}
\]
Here $\theta^{f_0,f_1}_{c,F}$ is the isomorphism class of $\theta^{f_0,f_1}_{U,F}$, and it depends only on the choice of $c$, and its stable isomorphism class is $\theta^{f_0,f_1}_{W,F}$, which is independent of choices. The equivalence class $[\theta^{f_0,f_1}_{W,F}]$ only depends on the normal bordism class of $F$. 

There are two new results that are specific to the two-sided version. First, we show that $\theta^{f_0,f_1}_{U,F}$ is proper (and hence the same is true for any two-sided quasi-formation equivalent to it). Second, when $W$ is an h-cobordism, equivalently, $\theta^{f_0,f_1}_{W,F}$ is elementary, the isomorphisms $E_q(M_0,f_0) \rightarrow E_q(M_1,f_1)$ induced by $W$ and $\theta^{f_0,f_1}_{W,F}$ coincide (see Definitions \ref{def:iw} and \ref{def:ix}). This is what allows us to study the realisability of isomorphisms using the two-sided obstruction.

\subsection{Definitions, bordism invariance and realisation of stable isomorphism classes}

\begin{defin}
Let $c$ be a CW-decomposition of $W$ and $U \subset \interior W$ be a regular neighbourhood of an embedding $\sk_q W \rightarrow \interior W$ that is isotopic to the inclusion. We define $\theta^{f_0,f_1}_{U,F} = (\UM; L, V_0, V_1, p_0, p_1) \in \gc^T_{2q+1}(B,\xi,f_0,f_1)$ as follows. 

The obstruction $\theta_{U,F} = (\UM; L, V) \in \gc^T_{2q+1}(B,\xi)$ was defined in \cite[Section 6.1]{csn-qfc26}, and we take the same $\UM = (M, \lambda, \mu)$, $L$ and $V_0 = V$. Let $p_0 : \UV_0 \rightarrow E_q(M_0,f_0)$ be the surjective morphism defined in \cite[Lemma 7.3]{csn-qfc26}. As shown in the proof of \cite[Theorem 7.6]{csn-qfc26}, the same constructions applied to $-W$ (that is, $W$ with reversed orientation) yield another free half-rank direct summand $V_1 = V'$ with a surjective morphism $p_1 : \UV_1 \rightarrow -E_q(M_1,f_1)^*$. 
\end{defin}

The proof of \cite[Theorem 7.6]{csn-qfc26} also shows that $\lambda(x,y)=0$ for every $x \in V_0$, $y \in V_1$, equivalently, $V_{1-i} \leq V_i^{\perp}$. Since $V_i \leq M$ is a half-rank subgroup and $\lambda$ is nonsingular, $V_i^{\perp} \leq M$ is a half-rank direct summand containing $\Tor M$. It follows that $V_i^{\perp} = V_{1-i} \oplus \Tor M$, so $(\UM; L, V_0, V_1, p_0, p_1)$ satisfies all of the conditions in Definition \ref{def:2sqf}.

\begin{prop} \label{prop:obstr-proper}
The two-sided quasi-formation $\theta^{f_0,f_1}_{U,F}$ is proper.
\end{prop}

\begin{proof}
The map $p_i$ is surjective by \cite[Lemma 7.3]{csn-qfc26}. 

We will prove that $\Ker p_0 = \Ker p_1 \leq V_0 \cap V_1$ by showing that $\Ker p_i$ is the image of the boundary map $\partial : H_{q+1}(W \setminus \interior U, \partial U) \rightarrow H_q(\partial U)$ for both $i$. Consider the commutative diagram 
\[
\xymatrix{
H_{q+1}(W \setminus \interior U, \partial U) \ar[r]^-{j_i} \ar[dr]_{\partial} & H_{q+1}(W \setminus \interior U, \partial U \sqcup M_i) \ar[r]^-{d_i} \ar[d]^{d'_i} & H_q(M_i) \\
 & H_q(\partial U) & 
}
\]
in which the horizontal sequence is exact. The map $d'_i$ is injective with image $V_i$, and $p_i = d_i \circ (d'_i)^{-1}$ (see \cite[Theorem 6.10, Definition 6.7 and Lemma 7.3]{csn-qfc26}). Hence $\Ker p_i = d'_i(\Ker d_i) = d'_i(\Image j_i) = \Image(d'_i \circ j_i) = \Image \partial$, as required. 
\end{proof}

\begin{rem}
In fact $\Ker p_i = V_0 \cap V_1$, but we will not use this.
\end{rem}

In the following we consider the analogues of the results in \cite[Sections 6.2--6.5]{csn-qfc26}. All of the proofs can be adapted to the two-sided setting, here we will only describe the additional steps required to take care of the maps $p_i$.

\begin{defin}
Suppose that $c$ is a CW-decomposition of $W$. Let $\theta^{f_0,f_1}_{c,F} \in \gs^T_{2q+1}(B,\xi,f_0,f_1)$ denote the isomorphism class of $\theta^{f_0,f_1}_{U,F}$ for any regular neighbourhood $U$ of any embedding $\sk_q W \rightarrow \interior W$ isotopic to the inclusion $\sk_q W \rightarrow W$.
\end{defin}

\begin{lem} \label{lem:cwhe-isom}
Let $c_1$ and $c_2$ be two CW-decompositions of $W$, and let $K_i \subset W$ denote the $q$-skeleton of $W$ with respect to $c_i$, for $i=1,2$. Let $U_i$ be a regular neighbourhood of an embedding $K_i \rightarrow \interior W$ isotopic to the inclusion $K_i \rightarrow W$. Suppose that there is a homotopy equivalence $j : K_1 \rightarrow K_2$ that commutes (up to homotopy) with the inclusions $K_i \rightarrow W$. Then $\theta^{f_0,f_1}_{U_1,F} \cong \theta^{f_0,f_1}_{U_2,F}$, via an isomorphism $I : H_q(\partial U_1) \rightarrow H_q(\partial U_2)$ that fits into the commutative diagram of \cite[Lemma 6.16]{csn-qfc26}.
\end{lem}

\begin{proof}
We define $U_2'$, $g$ and $G$ as in \cite[Lemma 6.16]{csn-qfc26}, and write $(\UM; L, V_0, V_1, p_0, p_1) = \theta^{f_0,f_1}_{U_1,F}$, $(\UMp; L', V'_0, V'_1, p'_0, p'_1) = \theta^{f_0,f_1}_{U_2',F}$ and $(\UMpp; L'', V''_0, V''_1, p''_0, p''_1) = \theta^{f_0,f_1}_{U_2,F}$. We will show that the isomorphisms $(\UM; L, V_0) \cong (\UMp; L', V'_0) \cong (\UMpp; L'', V''_0)$ defined in \cite[Lemma 6.16]{csn-qfc26} are also isomorphisms between $\theta^{f_0,f_1}_{U_1,F}$, $\theta^{f_0,f_1}_{U_2',F}$ and $\theta^{f_0,f_1}_{U_2,F}$.

The diffeomorphism $G$ (constructed via isotopy extension) can be assumed to be the identity on $\partial W = M_0 \sqcup M_1$. Thus it induces a commutative diagram
\[
\xymatrix{
H_q(M_0) \ar@{=}[d] & H_{q+1}(W \setminus \interior U_1, \partial U_1 \sqcup M_0) \ar[r] \ar[l] \ar[d]_{\cong} & H_q(\partial U_1) \ar[d]^{\cong} \\
H_q(M_0) & H_{q+1}(W \setminus \interior U_2', \partial U_2' \sqcup M_0) \ar[r] \ar[l] & H_q(\partial U_2') 
}
\]
Inclusions also induce a commutative diagram
\[
\xymatrix{
H_q(M_0) \ar@{=}[d] & H_{q+1}(W \setminus \interior U_2', \partial U_2' \sqcup M_0) \ar[r] \ar[l] \ar[d]_{\cong} & H_q(\partial U_2') \ar[d]^{\cong} \\
H_q(M_0) \ar@{=}[d] & H_{q+1}(W \setminus \interior U_2', C \sqcup M_0) \ar[r] \ar[l] & H_q(C) \\
H_q(M_0) & H_{q+1}(W \setminus \interior U_2, \partial U_2 \sqcup M_0) \ar[r] \ar[l] \ar[u]^{\cong} & H_q(\partial U_2) \ar[u]_{\cong} 
}
\]
It follows that the isomorphisms $(\UM; L, V_0) \cong (\UMp; L', V'_0) \cong (\UMpp; L'', V''_0)$ are also compatible with $p_0$, $p'_0$ and $p''_0$. Since $V_1$ and $p_1$ are defined in the same way as $V_0$ and $p_0$, but using $-W$ instead of $W$, the analogous diagrams show that they are also preserved. 
\end{proof}

\begin{defin}
Let $\theta^{f_0,f_1}_{W,F} \in \s^T_{2q+1}(B,\xi,f_0,f_1)$ be the stable isomorphism class of $\theta^{f_0,f_1}_{c,F}$ for any CW-decomposition $c$ of $W$. 
\end{defin}

\begin{thm} \label{thm:tw-welldef}
$\theta^{f_0,f_1}_{W,F}$ is well-defined. 
\end{thm}

\begin{proof}
Follows the proof of \cite[Theorem 6.20]{csn-qfc26}.
\end{proof}

\begin{lem} \label{lem:cwstab-isom}
Let $c$ and $c'$ be CW-decompositions of $W$. If $c'$ is a $k$-fold stabilisation of $c$, then $\theta^{f_0,f_1}_{c',F} \cong \theta^{f_0,f_1}_{c,F} \oplus \HH_{2k}$. In particular, $\theta^{f_0,f_1}_{c',F} \si \theta^{f_0,f_1}_{c,F}$ for any stabilisation $c'$ of $c$.
\end{lem}

\begin{proof}
As in \cite[Lemma 6.22]{csn-qfc26}, we assume that $c'$ is obtained from $c$ by a single stabilisation step, and define $K$, $K'$, $U$ and $U'$. Let $(\UM; L, V_0, V_1, p_0, p_1) = \theta^{f_0,f_1}_{U,F}$ and $(\UMp; L', V'_0, V'_1, p'_0, p'_1) = \theta^{f_0,f_1}_{U',F}$. We saw that $(\UMp; L', V'_0) \cong (\UM; L, V_0) \oplus (\UH_2; \{ 0 \} \times \Z, \Z \times \{ 0 \})$. The embedded $D^{q+1}$ used to construct the stabilisation $c'$ from $c$ also represents an element of $H_{q+1}(W \setminus \interior U', \partial U' \sqcup M_1)$, hence, similarly to the case of $V'_0$, we have $(0,(0,1)) \in V'_1$ and deduce that $V'_1 = V_1 \oplus (\Z \times \{ 0 \})$. Moreover, since the images of the homology class represented by $D^{q+1}$ under the boundary maps $H_{q+1}(W \setminus \interior U', \partial U' \sqcup M_0) \rightarrow H_q(M_0)$ and $H_{q+1}(W \setminus \interior U', \partial U' \sqcup M_1) \rightarrow H_q(M_1)$ are both trivial, we have $p'_0 \big| _{\Z \times \{ 0 \}} = p'_1 \big| _{\Z \times \{ 0 \}} = 0$. Therefore the isomorphism of \cite[Lemma 6.22]{csn-qfc26} is also an isomorphism $\theta^{f_0,f_1}_{U',F} \cong \theta^{f_0,f_1}_{U,F} \oplus \HH_2$.
\end{proof}

\begin{prop} \label{prop:theta-cw-real}
If $\theta^{f_0,f_1}_{W,F} \si x$ for some $x \in \gs^T_{2q+1}(B,\xi,f_0,f_1)$, then there is a CW-decomposition $c$ of $W$ and an integer $k \geq 0$ such that $\theta^{f_0,f_1}_{c,F} \cong x \oplus \HH_{2k}$.
\end{prop}

\begin{proof}
Follows the proof of \cite[Proposition 6.25]{csn-qfc26}.
\end{proof}

\begin{thm} \label{thm:theta-bordism-inv}
Suppose that $F' : W' \rightarrow B$ is another $q$-connected normal bordism between $f_0$ and $f_1$ over $(B,\xi)$. If $F'$ is normally bordant to $F$ (rel boundary), then $[\theta^{f_0,f_1}_{W,F}] = [\theta^{f_0,f_1}_{W',F'}] \in \ell^T_{2q+1}(B,\xi,f_0,f_1)$. 
\end{thm}

\begin{proof}
As in \cite[Theorem 6.26]{csn-qfc26}, we assume that $W'$ is obtained from $W$ by a single $q$-surgery, and define $c_0$,$c$, $c'$, $U_0$ $U$, $U'$, $\UM_0$ and $L_0$. Let $(\UM; L, V_0, V_1, p_0, p_1) = \theta^{f_0,f_1}_{U,F}$ and $(\UMp; L', V'_0, V'_1, p'_0, p'_1) = \theta^{f_0,f_1}_{U',F}$. We saw that $\partial U' = \partial U \approx \partial U_0 \# (S^q \times S^q)$ and hence $\UMp = \UM \cong \UM_0 \oplus \UH_2$, and if $V_0$ is identified with its image in $H_q(\partial U_0 \# (S^q \times S^q))$, then $(\UM; L, V_0) \cong (\UM_0 \oplus \UH_2; L_0 \oplus (\{ 0 \} \times \Z), V_0)$ and $(\UMp; L', V'_0) \cong (\UM_0 \oplus \UH_2; L_0 \oplus (\Z \times \{ 0 \}), V_0)$. In particular, the reason that $V'_0$ corresponds to the same subgroup of $H_q(\partial U_0 \# (S^q \times S^q))$ as $V_0$ is that $W \setminus \interior U = W' \setminus \interior U'$. Since $V_1$, $p_0$, $p_1$ and $V'_1$, $p'_0$, $p'_1$ also depend only on $W \setminus \interior U$ and $W' \setminus \interior U'$, the previous identifications also determine isomorphisms $(\UM; L, V_0, V_1, p_0, p_1) \cong (\UM_0 \oplus \UH_2; L_0 \oplus (\{ 0 \} \times \Z), V_0, V_1, p_0, p_1)$ and $(\UMp; L', V'_0, V'_1, p'_0, p'_1) \cong (\UM_0 \oplus \UH_2; L_0 \oplus (\Z \times \{ 0 \}), V_0, V_1, p_0, p_1)$. Hence $\theta^{f_0,f_1}_{W,F} \sim \theta^{f_0,f_1}_{W',F'}$.
\end{proof}

\begin{thm} \label{thm:theta-equiv-real}
If $\theta^{f_0,f_1}_{W,F} \sim x$ for some $x \in \s^T_{2q+1}(B,\xi,f_0,f_1)$, then there is a $q$-connected normal bordism $F' : W' \rightarrow B$ that is normally bordant to $F$ (rel boundary) such that $\theta^{f_0,f_1}_{W',F'} \si x$.
\end{thm}

\begin{proof}
Follows the proof of \cite[Theorem 6.27]{csn-qfc26}.
\end{proof}

\begin{lem} \label{lem:theta-surg-real}
Suppose that $\theta^{f_0,f_1}_{W,F} \si (\UM_0 \oplus \UH_2; L_0 \oplus ( \{0 \} \times \Z), V_0, V_1, p_0, p_1)$ for some $\UM_0$, $L_0$, $V_0$, $V_1$, $p_0$ and $p_1$. Then there is a $q$-connected normal bordism $F'' : W'' \rightarrow B$ which is normally bordant to $F$ such that $\theta^{f_0,f_1}_{W'',F''} \si (\UM_0 \oplus \UH_2; L_0 \oplus (\Z \times \{0 \} ), V_0, V_1, p_0, p_1)$.
\end{lem}

\begin{proof}
We construct the normal bordism $F''$ from $F$ by applying the same surgery as in \cite[Lemma 6.28]{csn-qfc26}. We also use the notation established there, except for writing $V_0 = V$ and adding $V_1$, $p_0$ and $p_1$ to the quasi-formations (or using appropriate variants, eg.\ $\bar{V}_0 = \bar{V}$). Thus $\theta^{f_0,f_1}_{W,F} \si (\UbM_0 \oplus \UH_2; \bar{L}_0 \oplus ( \{0 \} \times \Z), \bar{V}_0, \bar{V}_1, \bar{p}_0, \bar{p}_1)$ and we need to show that $\theta^{f_0,f_1}_{W'',F''} \si (\UbM_0 \oplus \UH_2; \bar{L}_0 \oplus (\Z \times \{0 \} ), \bar{V}_0, \bar{V}_1, \bar{p}_0, \bar{p}_1)$.

We apply Theorem \ref{thm:theta-bordism-inv} and define $c'$, $c''$, $U'_0$, $\UMp_0$ and $L'_0$ as before, and we get that $\theta^{f_0,f_1}_{W,F} \si (\UMp_0 \oplus \UH_2; L'_0 \oplus (\{ 0 \} \times \Z), V'_0, V'_1, p'_0, p'_1)$ and $\theta^{f_0,f_1}_{W'',F''} \si (\UMp_0 \oplus \UH_2; L'_0 \oplus ( \Z \times \{0 \}), V'_0, V'_1, p'_0, p'_1)$ for some $V'_0$, $V'_1$, $p'_0$, $p'_1$. By Theorem \ref{thm:tw-welldef} $(\UMp_0 \oplus \UH_2; L'_0 \oplus (\{ 0 \} \times \Z), V'_0, V'_1, p'_0, p'_1)$ and $(\UbM_0 \oplus \UH_2; \bar{L}_0 \oplus ( \{0 \} \times \Z), \bar{V}_0, \bar{V}_1, \bar{p}_0, \bar{p}_1)$ are stably isomorphic. Note that Theorem \ref{thm:tw-welldef} constructs the same stable isomorphism from a given $J$ (used in \cite[Lemma 6.23]{csn-qfc26}) as \cite[Theorem 6.20]{csn-qfc26}. The proof of \cite[Lemma 6.28]{csn-qfc26} shows that $J$ can be chosen such that the stable isomorphism restricts to an isomorphism between (the stabilisations of) $L'_0 \oplus ( \Z \times \{0 \})$ and $\bar{L}_0 \oplus (\Z \times \{0 \} )$. Therefore it is also a stable isomorphism $(\UMp_0 \oplus \UH_2; L'_0 \oplus ( \Z \times \{0 \}), V'_0, V'_1, p'_0, p'_1) \si (\UbM_0 \oplus \UH_2; \bar{L}_0 \oplus (\Z \times \{0 \} ), \bar{V}_0, \bar{V}_1, \bar{p}_0, \bar{p}_1)$.
\end{proof}

\subsection{Elementary obstructions and h-cobordisms}

Now we show that if $W$ is an h-cobordism, then $\theta^{f_0,f_1}_{W,F}$ is elementary, and it induces the same isomorphism $E_q(M_0,f_0) \rightarrow E_q(M_1,f_1)$ as $W$.

\begin{defin} \label{def:iw}
Suppose that $W$ is an h-cobordism. The \emph{isomorphism induced by $W$}, denoted $I_W : E_q(M_0,f_0) \rightarrow E_q(M_1,f_1)$, is given by the composition $H_q(M_0) \rightarrow H_q(W) \rightarrow H_q(M_1)$ of the isomorphism $H_q(M_i) \xra{\cong} H_q(W)$ induced by the inclusion (for $i=0$) and its inverse (for $i=1$).
\end{defin}

\begin{prop} \label{prop:hcob-elem}
$W$ is an h-cobordism if and only if $\theta^{f_0,f_1}_{W,F} \in \s^T_{2q+1}(B,\xi,f_0,f_1)$ is elementary. In this case $I_W = I_{\theta^{f_0,f_1}_{W,F}}$.
\end{prop}

\begin{proof}
By \cite[Proposition 6.30]{csn-qfc26} $W$ is an h-cobordism if and only if $\theta_{W,F}$ is elementary, which is equivalent to $\theta^{f_0,f_1}_{W,F}$ being elementary. 

Suppose that this is the case, then $\theta^{f_0,f_1}_{W,F}$ is both elementary and proper (see Proposition \ref{prop:obstr-proper}), so $I_{\theta^{f_0,f_1}_{W,F}}$ is defined. Consider the exact sequence 
\[
\ldots \rightarrow H_{q+1}(W, M_0 \sqcup M_1) \xra{\alpha} H_q(M_0) \oplus H_q(M_1) \xra{\beta} H_q(W) \rightarrow \ldots
\]
By the definition of $I_W$, we have $\Image(\alpha) = \Ker(\beta) = \Delta^*_{I_W}$ (where $\Delta^*_{I_W}$ denotes the anti-diagonal determined by $I_W$, see \cite[Definition 2.17]{csn-qfc26}).

Now suppose that $x \in H_q(M_0)$. Since $p_0$ is surjective, there is an $x' \in V_0$ such that $x = p_0(x')$. Let $y' \in V_1$ be the image of $x'$ under projection along $L$, so that $y' = x'+l$ for some $l \in L$. Let $y = p_1(y') \in H_q(M_1)$. Then by definition $I_{\theta^{f_0,f_1}_{W,F}}(x)=y$. 

The element $l \in L$ is the image of some $\bar{l} \in H_{q+1}(U, \partial U)$. Since $x = p_0(x')$ and $y = p_1(y')$, the elements $x'-x \in H_q(\partial U \sqcup M_0)$ and $y-y' \in H_q(\partial U \sqcup M_1)$ are the boundaries of some $\bar{x} \in H_{q+1}(W \setminus \interior U, \partial U \sqcup M_0)$ and $\bar{y} \in H_{q+1}(W \setminus \interior U, \partial U \sqcup M_1)$. The sum of appropriate relative cycles representing $\bar{x}$, $\bar{l}$ and $\bar{y}$ represents a homology class in $H_{q+1}(W, M_0 \sqcup M_1)$ with boundary $y-x \in H_q(M_0 \sqcup M_1)$, showing that $(-x,y) \in \Image(\alpha) = \Delta^*_{I_W}$. Therefore $I_W(x) = y$. 
\end{proof}

\begin{prop} \label{prop:bord-hcob-elem}
$F$ is normally bordant (rel boundary) to some normal bordism $F' : W' \rightarrow B$ such that $W'$ is an h-cobordism if and only if $[\theta^{f_0,f_1}_{W,F}] \in \ell^T_{2q+1}(B,\xi,f_0,f_1)$ is elementary.
\end{prop}

\begin{proof}
Follows the proof of \cite[Proposition 6.30]{csn-qfc26}.
\end{proof}

\section{Realising isomorphisms when $H_{q-1}(B)$ is free} \label{s:real-free}

\begin{proof}[Proof of Theorem \ref{thm:main-free}]
First we apply surgery below the middle dimension (see \cite[Proposition 4]{kreck99}) and replace $F_0$ with a $q$-connected normal bordism $F : W \rightarrow B$, which is normally bordant to $F_0$.

Note that the existence of $I$ implies that $E_q(M_0, f_0) \oplus (-E_q(M_1, f_1)^*)$ is metabolic (the anti-diagonal $\Delta^*_I$ is a lagrangian), so the conditions of \cite[Theorem 7.6]{csn-qfc26} are satisfied. We will use the notation and facts established in \cite[Theorem 7.6]{csn-qfc26}, except for writing $V_0$ for $V$, $V_1$ for $V'$, $p_0$ for $p$, etc. We will also use the notation $\UE_i = (E_i, \lambda_i, \mu_i)$ for $E_q(M_i, f_i)$. 

Since $H_{q-1}(B)$ is free, $\Tor H_q(\partial U) \cong \Tor H_{q-1}(B) \cong 0$ by \cite[Lemma 6.3]{csn-qfc26}. We also have $\Tor H_q(M_i) \cong \Tor H_{q-1}(M_i) \cong \Tor H_{q-1}(B)$ by Poincar\'e-duality and the fact that $f_i$ is $q$-connected. So all groups involved are free and there is no need to use free reductions. We have 
\[
\begin{aligned}
\theta^{f_0,f_1}_{W,F} &\si (\UM; L, V_0, V_1, p_0, p_1) \sim \\
 &\sim (\UM; L, (i_0 \oplus i_1)(\Delta^*_I) \oplus Z, (i_0 \oplus i_1)(\Delta^*_I) \oplus Z, 0, 0) \oplus (\UM; (i_0 \oplus i_1)(\Delta^*_I) \oplus Z, V_0, V_1, p_0, p_1) = \\
 &= x \oplus (\UM; (i_0 \oplus i_1)(\Delta^*_I) \oplus Z, V_0, V_1, p_0, p_1) = \\
 &= x \oplus (\UN \oplus \UN'; (i_0 \oplus i_1)(\Delta^*_I) \oplus Z, Y_0 \oplus Z, Y_1 \oplus Z, p_0, p_1) = \\
 &= x \oplus (\UN; (i_0 \oplus i_1)(\Delta^*_I), Y_0, Y_1, p_0 \big| _{Y_0}, p_1 \big| _{Y_1}) \oplus (\UN'; Z, Z, Z, 0, 0) \sim \\
 &\sim x \oplus (\UN; (i_0 \oplus i_1)(\Delta^*_I), Y_0, Y_1, p_0 \big| _{Y_0}, p_1 \big| _{Y_1}) \cong \\
 &\cong x \oplus (\UE_0 \oplus (-\UEs_1); \Delta^*_I, E_0, E_1, \id_{E_0}, \id_{E_1})
\end{aligned}
\]
where $x = [\UM; L, (i_0 \oplus i_1)(\Delta^*_I) \oplus Z, (i_0 \oplus i_1)(\Delta^*_I) \oplus Z, 0, 0] \in L_{2q+1}(B,\xi,f_0,f_1)$, and we used Theorem \ref{thm:jacobi}, the facts that $p_j \circ i_j = \id_{E_j}$ and $\Ker p_j = 0$, Lemma \ref{lem:nkk-zero} and the isomorphism $i_0 \oplus i_1$.

Let $e = (\UE_0 \oplus (-\UEs_1); \Delta^*_I, E_0, E_1, \id_{E_0}, \id_{E_1})$, so that $\theta^{f_0,f_1}_{W,F} \sim x \oplus e$. By Proposition \ref{prop:L-forget} b) $L_{2q+1}(B,\xi,f_0,f_1) \cong 0$, so $x$ is the zero element, and by Proposition \ref{prop:L-action} $x \oplus e \sim e$. 

By Definition \ref{def:proper} (or using that it is equivalent to $\theta^{f_0,f_1}_{W,F}$), the two-sided quasi-formation $e$ is proper. It is also elementary and $I_e = I$, because $\Delta^*_I$ is the anti-diagonal determined by $I$. By Theorem \ref{thm:theta-equiv-real} there is a normal bordism $F' : W' \rightarrow B$, normally bordant to $F$, such that $\theta^{f_0,f_1}_{W',F'} \si e$. By Proposition \ref{prop:hcob-elem} $W'$ is an h-cobordism and $I_{W'} = I_e = I$. 
\end{proof}

\section{Necessary conditions}

Now we move on to proving Theorems \ref{thm:main-torsion} and \ref{thm:b-c-intro}, and determining which isomorphisms $I : E_q(M_0, f_0) \rightarrow E_q(M_1, f_1)$ can be realised by h-cobordisms over $B$, without the assumption that $H_{q-1}(B)$ is free. In this section we will describe necessary conditions. First we consider systems of Q-forms. Then, in the case when $I$ is an automorphism, we also give explicit conditions on the entries of the matrix form of $I$.

\subsection{Systems of Q-forms} \label{ss:sys}

To a map $M \rightarrow B$ we can associate not only the Q-form introduced earlier, but also its analogues with $\Z_m$ coefficients, which are compatible with the integral Q-form. Below we set up the terminology and notation needed to treat these.

\begin{defin}
Let $R$ be a $\Z_m$-module for some $m \geq 2$. An \emph{extended quadratic form over $R$} is a triple
\[
\UM = (M, \lambda, \mu)
\]
where $M$ is a $\Z_m$-module, $\lambda : M \times M \rightarrow \Z_m$ is a symmetric bilinear function, and $\mu : M \rightarrow R$ is a $\Z_m$-module homomorphism. 
\end{defin}

\begin{defin} 
Let $\UM_i = (M_i, \lambda_i, \mu_i)$ be an extended quadratic form over a $\Z_m$-module $R_i$, for some $m \geq 2$ and $i = 1,2$, and let $r : R_1 \rightarrow R_2$ be a homomorphism. A \emph{morphism} $h : \UM_1 \rightarrow \UM_2$ over $r$ is a homomorphism $h : M_1 \rightarrow M_2$ of $\Z_m$-modules such that $\lambda_2(h(x),h(y)) = \lambda_1(x,y)$ and $\mu_2(h(x)) = r(\mu_1(x))$ for every $x,y \in M_1$. A \emph{morphism} between extended quadratic forms over a fixed $\Z_m$-module $R$ is a morphism over $\id_R$. 
\end{defin}

\begin{defin}
Suppose that $\UM = (M, \lambda, \mu)$ is an extended quadratic form over an abelian group $Q$. Let $m \geq 2$. The \emph{mod $m$ reduction} of $\UM$ is the extended quadratic form 
\[
\varrho_m \UM = (M \otimes \Z_m, \varrho_m \lambda, \varrho_m \mu)
\]
over the $\Z_m$-module $Q \otimes \Z_m$, where $\varrho_m \lambda(x \otimes 1, y \otimes 1) = \varrho_m(\lambda(x, y)) \in \Z_m$ for every $x,y \in M$, and $\varrho_m \mu = \mu \otimes \id_{\Z_m} : M \otimes \Z_m \rightarrow Q \otimes \Z_m$. 

If $I : \UM = (M, \lambda, \mu) \rightarrow \UMp = (M', \lambda', \mu')$ is a morphism of extended quadratic forms over an abelian group $Q$, then $\varrho_m I : \varrho_m \UM \rightarrow \varrho_m \UMp$ is the morphism of extended quadratic forms over the $\Z_m$-module $Q \otimes \Z_m$ given by $I \otimes \id_{\Z_m} : M \otimes \Z_m \rightarrow M' \otimes \Z_m$.
\end{defin}

\begin{defin}
A \emph{system of modules} is a collection $\QQ = (Q, \{ Q_m \} , \{ r_m \} )$ consisting of an abelian group $Q$, and for every integer $m \geq 2$, a $\Z_m$-module $Q_m$ and a $\Z_m$-module homomorphism $r_m : Q \otimes \Z_m \rightarrow Q_m$. 
\end{defin}

\begin{defin}
A \emph{system of extended quadratic forms} over a system of modules $\QQ = (Q, \{ Q_m \} , \{ r_m \} )$ is a collection $\MM = (\UM, \{ \UM_m \} , \{ \iota_m \} )$ consisting of an extended quadratic form $\UM$ over $Q$, and for every integer $m \geq 2$, an extended quadratic form $\UM_m$ over $Q_m$ and a morphism $\iota_m : \varrho_m \UM \rightarrow \UM_m$ over $r_m$. 
\end{defin}

\begin{defin}
Suppose that $\MM = (\UM, \{ \UM_m \} , \{ \iota_m \} )$ and $\MM' = (\UMp, \{ \UMp_m \} , \{ \iota'_m \} )$ are systems of extended quadratic forms over a system of modules $\QQ = (Q, \{ Q_m \} , \{ r_m \} )$. A \emph{morphism} $\II : \MM \rightarrow \MM'$ is a collection $\II = (I, \{ I_m \} )$ consisting of a morphism $I : \UM \rightarrow \UMp$, and for every integer $m \geq 2$, a morphism $I_m : \UM_m \rightarrow \UMp_m$ such that the diagram 
\[
\xymatrix{
\varrho_m \UM \ar[r]^-{\varrho_m I} \ar[d]_{\iota_m} & \varrho_m \UMp \ar[d]^{\iota'_m} \\
\UM_m \ar[r]^-{I_m} & \UMp_m
}
\]
which consists of morphisms over the diagram 
\[
\xymatrix{
Q \otimes \Z_m \ar[r]^-{\id} \ar[d]_{r_m} & Q \otimes \Z_m \ar[d]^{r_m} \\
Q_m \ar[r]^-{\id} & Q_m
}
\]
of $\Z_m$-modules, commutes.
\end{defin}

\begin{defin}
Suppose that $B$ is a topological space and $q$ is a positive integer. The \emph{system of modules determined by $B$} is 
\[
\HH_q(B) = (H_q(B), \{ H_q(B; \Z_m) \} , \{ r_m \} )
\]
where $r_m$ is the inclusion in the universal coefficient exact sequence
\[
0 \rightarrow H_q(B) \otimes \Z_m \xra{r_m} H_q(B; \Z_m) \rightarrow \Tor(H_{q-1}(B), \Z_m) \rightarrow 0
\]
\end{defin}

\begin{defin}
Suppose that $B$ is a topological space, $q \geq 2$ is an even integer and $m \geq 2$ is an integer. Let $M$ be a closed oriented $2q$-manifold with a map $f : M \rightarrow B$. The \emph{mod $m$ Q-form} of $f$ is the extended quadratic form
\[
E_q(M,f)_m = (H_q(M; \Z_m), \lambda_m, \mu_m)
\] 
over $H_q(B; \Z_m)$, where $\lambda_m : H_q(M; \Z_m) \times H_q(M; \Z_m) \rightarrow \Z_m$ is the mod $m$ intersection form of $M$ and $\mu_m = H_q(f; \Z_m) : H_q(M; \Z_m) \rightarrow H_q(B; \Z_m)$. 
\end{defin}

It follows from Poincar\'e-duality that $\lambda_m$ is nonsingular, in the sense that its adjoint is an isomorphism $\ad \lambda_m : H_q(M; \Z_m) \rightarrow \Hom(H_q(M; \Z_m), \Z_m)$.

\begin{defin}
Suppose that $B$ is a topological space and $q \geq 2$ is an even integer. Let $M$ be a closed oriented $2q$-manifold with a map $f : M \rightarrow B$. The \emph{system of Q-forms} of $f$ is the system of extended quadratic forms
\[
\EE_q(M,f) = (E_q(M,f), \{ E_q(M,f)_m \} , \{ \iota_m \} )
\] 
over $\HH_q(B)$, where $\iota_m : H_q(M) \otimes \Z_m \rightarrow H_q(M; \Z_m)$ is the inclusion.
\end{defin}

\begin{prop} \label{prop:ind-sys}
Suppose that $B$ is a topological space and $q \geq 2$ is an even integer. Let $M_i$ be a closed oriented $2q$-manifold with a map $f_i : M_i \rightarrow B$, for $i = 0,1$. Then an orientation-preserving homotopy equivalence $h : M_0 \rightarrow M_1$ over $B$ (ie.\ $f_1 \circ h \simeq f_0$) induces an isomorphism $\II : \EE_q(M_0,f_0) \rightarrow \EE_q(M_1,f_1)$ between the systems of Q-forms of $M_0$ and $M_1$, namely $\II = (H_q(h), \{ H_q(h; \Z_m) \} )$.
\qed
\end{prop}

\subsection{Integral conditions} \label{ss:int}

Suppose that $B$ is a topological space and $q \geq 2$ is an even integer. Let $M$ be a closed oriented $2q$-manifold with a map $f : M \rightarrow B$. Let $Q = H_q(B)$, and let $E_q(M,f) = (H_q(M), \lambda, \mu)$ be the Q-form of $f$. Let $\tp = \Tor H_q(M)$ and $\fp = H_q(M) / \Tor H_q(M)$ be the torsion part and free part of $H_q(M)$ respectively, and fix a splitting $H_q(M) = \fp \oplus \tp$. 

Assume that $\mu \big| _{\tp} = 0$. Then $E_q(M,f)$ has a free reduction, which we will denote by $\overline{E_q(M,f)} = (\fp, \lambda \big| _{\fp \times \fp}, \mu \big| _{\fp})$. Also, the map $\mu \big| _{\fp} : \fp \rightarrow Q$ is independent of the choice of splitting of $H_q(M)$. 

If $I$ is an automorphism of $E_q(M,f)$, then it can be written in the form $I = \bigl[ \begin{smallmatrix} a & 0 \\ b & c \end{smallmatrix} \bigr] : \fp \oplus \tp \rightarrow \fp \oplus \tp$ for some automorphism $a$ of $\overline{E_q(M,f)}$, homomorphism $b : \fp \rightarrow \tp$ and automorphism $c$ of $\tp$. By Proposition \ref{prop:ind-sys}, if $I$ can be realised by an h-cobordism, then it has to extend to an automorphism $\II$ of $\EE_q(M,f)$. We will describe necessary conditions for the existence of such an $\II$ in terms of $I$.

\begin{defin} \label{def:can}
Consider the isomorphisms $\Hom(\fp, Q) \cong \Hom(\fp^*, Q) \cong \fp \otimes Q$. The first one is given by the isomorphism $\fp \cong \fp^*$ induced by $\ad \lambda$, and the inverse of the second one is given by $x \otimes q \mapsto (\varphi \mapsto \varphi(x)q)$ for $x \in \fp$, $q \in Q$ and $\varphi \in \fp^*$. The elements corresponding to $\mu \big| _{\fp} \in \Hom(\fp, Q)$ will be denoted by $\tilde{\mu} \in \Hom(\fp^*, Q)$ and $\hat{\mu} \in \fp \otimes Q$.
\end{defin}

\begin{thm} \label{thm:b-mu}
Suppose that $H_{q-1}(f) : H_{q-1}(M) \rightarrow H_{q-1}(B)$ is an isomorphism. Let $b : \fp \rightarrow \tp$ be a homomorphism and $I = \bigl[ \begin{smallmatrix} \id & 0 \\ b & \id \end{smallmatrix} \bigr] : \fp \oplus \tp \rightarrow \fp \oplus \tp$. If $I$ extends to an automorphism $\II$ of $\EE_q(M,f)$, then $(b \otimes \id_Q)(\hat{\mu}) = 0 \in \tp \otimes Q$. 
\end{thm}

\begin{proof}
We will prove that $(\xi \otimes \eta) \circ (b \otimes \id_Q)(\hat{\mu}) = 0$ for any integer $m \geq 2$ and homomorphisms $\xi : \tp \rightarrow \Z_m$ and $\eta : Q \rightarrow \Z_m$. This suffices, because $\Hom(\tp \otimes Q, \Z_m) \cong \Hom(\tp, \Z_m) \otimes \Hom(Q, \Z_m)$, and every non-trivial element of a finitely generated abelian group $X$ is detected by a homomorphism $X \rightarrow \Z_m$ for some $m$. Fix an $m$, $\xi$ and $\eta$. 

Consider the diagram between the universal coefficient exact sequences induced by $f$. 
\[
\xymatrix{
0 \ar[r] & H_q(M) \otimes \Z_m \ar[r]^-{\iota_m} \ar[d]_{\mu \otimes \id} & H_q(M; \Z_m) \ar[r]^-{\pi^M_m} \ar[d]_{\mu_m} & \Tor(H_{q-1}(M), \Z_m) \ar[r] \ar[d]^{\cong} & 0 \\
0 \ar[r] & H_q(B) \otimes \Z_m \ar[r]^-{r_m} & H_q(B; \Z_m) \ar[r]^-{\pi^B_m} & \Tor(H_{q-1}(B), \Z_m) \ar[r] & 0
}
\]

The composition 
\[
H_q(M; \Z_m) \xra{\ad \lambda_m} \Hom(H_q(M; \Z_m), \Z_m) \xra{\iota_m^*} \Hom(H_q(M) \otimes \Z_m, \Z_m) \rightarrow \Hom(\tp \otimes \Z_m, \Z_m)
\] 
is surjective, because $\ad \lambda_m$ is an isomorphism, and $\tp \otimes \Z_m$ is a summand in $H_q(M) \otimes \Z_m$, which is a summand in $H_q(M; \Z_m)$ (by the splitting of the universal coefficient exact sequence). So there is an $x \in H_q(M; \Z_m)$ such that $\xi = \iota_m^* \circ (\ad \lambda_m)(x) \big| _{\tp \otimes \Z_m} \in \Hom(\tp, \Z_m) \cong \Hom(\tp \otimes \Z_m, \Z_m)$. 

Let $\II = (I, \{ I_m \} )$ be an automorphism of $\EE_q(M,f)$ that extends $I$. Since $I_m : H_q(M; \Z_m) \rightarrow H_q(M; \Z_m)$ is an automorphism of $E_q(M,f)_m$, we have $\mu_m = \mu_m \circ I_m$, hence $\pi^B_m \circ \mu_m = \pi^B_m \circ \mu_m \circ I_m$, hence $\pi^M_m = \pi^M_m \circ I_m$. In particular, $\pi^M_m(x) = \pi^M_m \circ I_m(x)$, so $I_m(x) - x \in \Ker(\pi^M_m) = \Image(\iota_m)$. Let $\alpha \in H_q(M) \otimes \Z_m$ be an element such that $\iota_m(\alpha) = I_m(x) - x$. Now $\mu_m(x) = \mu_m \circ I_m(x) = \mu_m(x) + \mu_m \circ \iota_m(\alpha) = \mu_m(x) + r_m \circ (\mu \otimes \id)(\alpha)$, and $r_m$ is injective, therefore $(\mu \otimes \id)(\alpha) = 0$. 

For every $y \in \fp \leq H_q(M)$ we have $\lambda_m(x, \iota_m(y \otimes 1)) = \lambda_m(I_m(x), I_m \circ \iota_m(y \otimes 1))$. Here $I_m(x) = x + \iota_m(\alpha)$ and $I_m \circ \iota_m(y \otimes 1) = \iota_m \circ (\varrho_m I)(y \otimes 1) = \iota_m \circ (I \otimes \id_{\Z_m})(y \otimes 1) = \iota_m(I(y) \otimes 1) = \iota_m((y+b(y)) \otimes 1) = \iota_m(y \otimes 1) + \iota_m(b(y) \otimes 1)$, using that $\II$ extends $I$. Thus the difference is $\lambda_m(\iota_m(\alpha), \iota_m(y \otimes 1)) + \lambda_m(\iota_m(\alpha), \iota_m(b(y) \otimes 1)) + \lambda_m(x, \iota_m(b(y) \otimes 1)) = 0$. 

Since $\iota_m : \varrho_m E_q(M,f) \rightarrow E_q(M,f)_m$ is a morphism, $\lambda_m(\iota_m(\alpha), \iota_m(y \otimes 1)) = (\varrho_m \lambda)(\alpha, y \otimes 1)$. We have $\lambda_m(\iota_m(\alpha), \iota_m(b(y) \otimes 1)) = (\varrho_m \lambda)(\alpha, b(y) \otimes 1) = 0$, because $b(y) \in \tp$ and $\tp \leq H_q(M)$ is the radical of $\lambda$. Finally, $\lambda_m(x, \iota_m(b(y) \otimes 1)) = \xi(b(y))$ by the definition of $x$. Therefore $(\varrho_m \lambda)(\alpha, y \otimes 1) = -\xi \circ b(y)$ for every $y \in \fp$. 

Since $H_q(M) = \fp \oplus \tp$, also $H_q(M) \otimes \Z_m = \fp \otimes \Z_m \oplus \tp \otimes \Z_m$, hence $\alpha = \alpha' + \alpha''$ for some $\alpha' \in \fp \otimes \Z_m$ and $\alpha'' \in \tp \otimes \Z_m$. Since $\alpha''$ is in the radical of $\varrho_m \lambda$, we have $(\varrho_m \lambda)(\alpha', y \otimes 1) = (\varrho_m \lambda)(\alpha, y \otimes 1) = -\xi \circ b(y)$ for every $y \in \fp$. Since $\fp$ is the free part of $H_q(M)$ and $\lambda$ is nonsingular, the adjoint of $\varrho_m \lambda$ determines an isomorphism between $\fp \otimes \Z_m$ and $\Hom(\fp \otimes \Z_m, \Z_m) \cong \Hom(\fp, \Z_m)$. By the previous equality, $\alpha' \in \fp \otimes \Z_m$ corresponds to $-\xi \circ b \in \Hom(\fp, \Z_m)$.

Let $\mu_{\eta} = \eta \circ \mu \big| _{\fp} : \fp \rightarrow \Z_m$. Analogously to Definition \ref{def:can}, we consider the isomorphisms $\Hom(\fp \otimes \Z_m, \Z_m) \cong \Hom(\Hom(\fp, \Z_m), \Z_m) \cong \fp \otimes \Z_m$, where the first one is induced by the isomorphism $\ad (\varrho_m \lambda) : \fp \otimes \Z_m \cong \Hom(\fp, \Z_m)$. The elements corresponding to $\mu_{\eta} \in \Hom(\fp, \Z_m) \cong \Hom(\fp \otimes \Z_m, \Z_m)$ will be denoted by $\tilde{\mu}_{\eta} \in \Hom(\Hom(\fp, \Z_m), \Z_m)$ and $\hat{\mu}_{\eta} \in \fp \otimes \Z_m$. In particular, $\hat{\mu}_{\eta} = (\id_{\fp} \otimes \eta)(\hat{\mu})$.

We saw earlier that $(\mu \otimes \id)(\alpha) = 0$. Since $\mu \big| _{\tp} = 0$, this implies that $(\mu \otimes \id)(\alpha') = 0$, and hence $(\mu_{\eta} \otimes \id)(\alpha') = 0$. Using the correspondences determined by $\ad (\varrho_m \lambda)$, this means that $\tilde{\mu}_{\eta}(-\xi \circ b) = 0$, hence $\tilde{\mu}_{\eta}(\xi \circ b) = 0$. By the definition of $\hat{\mu}_{\eta}$, this is equivalent to $((\xi \circ b) \otimes \id_{\Z_m})(\hat{\mu}_{\eta}) = 0$. Therefore $0 = ((\xi \circ b) \otimes \id_{\Z_m})(\hat{\mu}_{\eta}) = ((\xi \circ b) \otimes \id_{\Z_m}) \circ (\id_{\fp} \otimes \eta)(\hat{\mu}) = ((\xi \circ b) \otimes \eta)(\hat{\mu}) = (\xi \otimes \eta) \circ (b \otimes \id_Q)(\hat{\mu})$, as required. 
\end{proof}

\begin{prop} \label{prop:c-id}
Suppose that $H_{q-1}(f) : H_{q-1}(M) \rightarrow H_{q-1}(B)$ is an isomorphism. Let $I = \bigl[ \begin{smallmatrix} a & 0 \\ b & c \end{smallmatrix} \bigr] : \fp \oplus \tp \rightarrow \fp \oplus \tp$ be an automorphism of $E_q(M,f)$. If $I$ extends to an automorphism $\II$ of $\EE_q(M,f)$, then $c = \id_{\tp}$. 
\end{prop}

\begin{proof}
We will prove that $\xi(c(z) - z) = 0$ for every integer $m \geq 2$, homomorphism $\xi : \tp \rightarrow \Z_m$ and $z \in \tp$. Define $x$ and $\alpha$ in the same way as in the proof of Theorem \ref{thm:b-mu}.

We have $\lambda_m(x, \iota_m(z \otimes 1)) = \lambda_m(I_m(x), I_m \circ \iota_m(z \otimes 1))$. Here $I_m(x) = x + \iota_m(\alpha)$ and $I_m \circ \iota_m(z \otimes 1) = \iota_m \circ (\varrho_m I)(z \otimes 1) = \iota_m \circ (I \otimes \id_{\Z_m})(z \otimes 1) = \iota_m(I(z) \otimes 1) = \iota_m(c(z) \otimes 1)$, using that $\II$ extends $I$. So $0 = \lambda_m(x + \iota_m(\alpha), \iota_m(c(z) \otimes 1)) - \lambda_m(x, \iota_m(z \otimes 1)) = \lambda_m(\iota_m(\alpha), \iota_m(c(z) \otimes 1)) + \lambda_m(x, \iota_m((c(z) - z) \otimes 1))$. Now $\lambda_m(\iota_m(\alpha), \iota_m(c(z) \otimes 1)) = (\varrho_m \lambda)(\alpha, c(z) \otimes 1) = 0$, because $c(z) \in \tp$. Thus $0 = \lambda_m(x, \iota_m((c(z) - z) \otimes 1)) = \xi(c(z) - z)$, as required. 
\end{proof}

\begin{rem} \label{rem:tor}
It follows from Propositions \ref{prop:ind-sys} and \ref{prop:c-id} that if $I$ is induced by a homotopy automorphism of $M$ over $B$, then $I \big| _{\tp} = \id_{\tp}$. We can also prove this directly, without using the existence of $\II$ as an intermediate condition. In fact, we will consider the more general situation of an isomorphism $I : E_q(M_0, f_0) \rightarrow E_q(M_1, f_1)$, and show that if $I$ is realised by a homotopy equivalence $M_0 \rightarrow M_1$ over $B$, then $I \big| _{\Tor H_q(M_0)}$ has to be a certain fixed map, depending only on $f_0$ and $f_1$ (assuming that $H_{q-1}(f_i)$ is an isomorphism). 

Let $\Phi_i : \Ext(H_{q-1}(B),\Z) \rightarrow \Ext(H_{q-1}(M_i),\Z) \rightarrow \Ext(H^{q+1}(M_i),\Z) \rightarrow \Tor H_q(M_i)$ be the composition of the isomorphisms induced by $H_{q-1}(f_i)$, Poincar\'e-duality and the universal coefficient formula (applied for computing the integral homology of $M_i$ from its cohomology). Then a homotopy equivalence over $B$ induces a commutative diagram
\[
\xymatrix{
\Ext(H_{q-1}(B),\Z) \ar[r]^-{\cong} \ar@{=}[d] & \Ext(H_{q-1}(M_0),\Z) \ar[r]^-{\cong} & \Ext(H^{q+1}(M_0),\Z) \ar[r]^-{\cong} \ar[d] & \Tor H_q(M_0) \ar[d]^-{I} \\
\Ext(H_{q-1}(B),\Z) \ar[r]^-{\cong} & \Ext(H_{q-1}(M_1),\Z) \ar[r]^-{\cong} \ar[u] & \Ext(H^{q+1}(M_1),\Z) \ar[r]^-{\cong} & \Tor H_q(M_1) 
}
\]
between the two compositions, showing that if $I : E_q(M_0, f_0) \rightarrow E_q(M_1, f_1)$ is the induced isomorphism, then $I \big| _{\Tor H_q(M_0)} = \Phi_1 \circ \Phi_0^{-1}$. 
\end{rem}

\begin{ex} \label{ex:i-ext}
Suppose that $q > 2$ is even. Let $M_1$ be a $(q{-}1)$-connected $2q$-manifold, and denote the free group $H_q(M_1)$ by $\fp$. Let $\tp$ be a torsion group and define the $2q$-manifold $M_2$ as the boundary of the $(2q+1)$-dimensional trivial thickening of the Moore-space $M(\tp,q-1)$ (which can be constructed as a regular neighbourhood of an embedding $M(\tp,q-1) \rightarrow \R^{2q+1}$). Then $H_{q-1}(M_2) \cong H_q(M_2) \cong \tp$. Let $M = M_1 \# M_2$, then $H_q(M) \cong H_q(M_1) \oplus H_q(M_2) \cong \fp \oplus \tp$. By the Hurewicz-theorem the map $\pi_q(M_2) \rightarrow H_q(M_2)$ is surjective. Define $B$ by adding $(q+1)$-cells to $M$, glued along maps $S^q \rightarrow M_2 \setminus D^{2q} \rightarrow M$ representing generators of $H_q(M_2) \leq H_q(M)$, so that $H_q(B) \cong \fp$ is free. Let $f : M \rightarrow B$ be the inclusion, which is $q$-connected. (Since $M_2$ is stably parallelisable, there is a stable bundle $\xi$ over $B$ that restricts to $\nu_M$, so $f$ could be covered by a bundle map.)

Let $I = \bigl[ \begin{smallmatrix} a & 0 \\ b & c \end{smallmatrix} \bigr] : \fp \oplus \tp \rightarrow \fp \oplus \tp$ be an automorphism of $E_q(M,f)$. Then $a = \id_F$ (because $H_q(f) \big| _{\fp}$ is injective), while $b$ can be any homomorphism $\fp \rightarrow \tp$ and $c$ can be any automorphism of $\tp$. Since $\mu \big| _{\fp} = H_q(f) \big| _{\fp} : \fp \rightarrow H_q(B) = Q$ is an isomorphism, the same is true for $\tilde{\mu} : \fp^* \rightarrow Q$, and it follows that $(b \otimes \id_Q)(\hat{\mu}) = 0 \in \tp \otimes Q$ holds only for $b = 0$. So by Theorem \ref{thm:b-mu} and Proposition \ref{prop:c-id} the only automorphism $I$ that extends to an automorphism $\II$ of $\EE_q(M,f)$ (and can be realised by a diffeomorphism of $M$ over $B$) is $I = \id_{H_q(M)}$. 

This construction can be modified, by adding further $(q+1)$-cells to $B$, glued along generators of a direct summand $\fp' < \fp$, so that $H_q(B) \cong \fp / \fp'$. Then some, but not all, non-trivial homomorphisms $b : \fp \rightarrow \tp$ also satisfy $(b \otimes \id_Q)(\hat{\mu}) = 0$. 
\end{ex}

\section{Realising automorphisms}

In this section we show that the conditions established in Section \ref{ss:int} are sufficient for realising an automorphism $I$. First we consider the algebraic setting, and show that $I$ is induced by a specially constructed elementary proper two-sided quasi-formation. Then we prove that there is a normal bordism whose obstruction is that particular two-sided quasi-formation, and hence it is an h-cobordism inducing $I$.

\subsection{Realisation by two-sided quasi-formations}

We fix an abelian group $Q$, a homomorphism $v : Q \rightarrow \Z_2$, and an extended quadratic form $\UP = (P, \lambda_P, \mu_P)$ over $(Q,v)$. Let $\tp = \Tor P$ and $\fp = P / \Tor P$, and fix a splitting $P = \fp \oplus \tp$. 
Assume that $\mu_P \big| _{\tp} = 0$, then we can define the free reduction $\UbP = (\fp, \bar{\lambda}_P, \bar{\mu}_P)$ of $\UP$. 
Assume also that $\UP$ is nonsingular, and define the elements $\tilde{\mu}_P \in \Hom(\fp^*, Q)$ and $\hat{\mu}_P \in \fp \otimes Q$ corresponding to $\mu_P \big| _{\fp} = \bar{\mu}_P \in \Hom(\fp, Q)$ as in Definition \ref{def:can}.

\begin{lem} \label{lem:b-lift}
Let $0 \rightarrow K \xra{i} Z \xra{p} \tp \rightarrow 0$ be a free resolution of $\tp$, and let $b : \fp \rightarrow \tp$ be a homomorphism. Suppose that $\mu_P : P \rightarrow Q$ is surjective. If $(b \otimes \id_Q)(\hat{\mu}_P) = 0 \in \tp \otimes Q$, then there is a homomorphism $\bar{b} : \fp \rightarrow Z$ such that $p \circ \bar{b} = b$ and $(\bar{b} \otimes \id_Q)(\hat{\mu}_P) = 0 \in Z \otimes Q$.
\end{lem}

\begin{proof}
Since $\fp$ is free and $p$ is surjective, $b$ can be lifted to some homomorphism $\bar{b}_0 : \fp \rightarrow Z$ (ie.\ $p \circ \bar{b}_0 = b$), let $\hat{x} = (\bar{b}_0 \otimes \id_Q)(\hat{\mu}_P) \in Z \otimes Q$. Since the tensor product is right exact, it induces an exact sequence $K \otimes Q \xra{i \otimes \id} Z \otimes Q \xra{p \otimes \id} \tp \otimes Q \rightarrow 0$. We have $(p \otimes \id_Q)(\hat{x}) = (b \otimes \id_Q)(\hat{\mu}_P) = 0$, so $\hat{x} = (i \otimes \id_Q)(\hat{y})$ for some $\hat{y} \in K \otimes Q$. 

Let $\tilde{y} : K^* \rightarrow Q$ be the element corresponding to $\hat{y}$ under the isomorphism $\Hom(K^*,Q) \cong K \otimes Q$. Since $\mu_P$ is surjective and $\mu_P \big| _{\tp} = 0$, the map $\mu_P \big| _{\fp} = \bar{\mu}_P : \fp \rightarrow Q$, and hence $\tilde{\mu}_P : \fp^* \rightarrow Q$ are also surjective. So $\tilde{y}$ can be lifted to a map $\bar{y} : K^* \rightarrow F^*$ such that $\tilde{\mu}_P \circ \bar{y} = \tilde{y}$. Let $\bar{b} = \bar{b}_0 - i \circ \bar{y}^* : \fp \rightarrow Z$, then $p \circ \bar{b} = p \circ \bar{b}_0 - p \circ i \circ \bar{y}^* = p \circ \bar{b}_0 = b$. 

It remains to check that $(\bar{b} \otimes \id_Q)(\hat{\mu}_P) = 0$. The map $\bar{y}$ and its dual induce a commutative diagram
\[
\xymatrix{
F \otimes Q \ar[r]^-{\cong} \ar[d]_{\bar{y}^* \otimes \id} & \Hom(F^*,Q) \ar[d]^{\Hom(\bar{y}, \id)} \\
K \otimes Q \ar[r]^-{\cong} & \Hom(K^*,Q)
}
\]
between the identifications we used, and $\hat{\mu}_P \in F \otimes Q$ corresponds to $\tilde{\mu}_P : \fp^* \rightarrow Q$, so $(\bar{y}^* \otimes \id_Q)(\hat{\mu}_P) \in K \otimes Q$ corresponds to $\Hom(\bar{y}, \id_Q)(\tilde{\mu}_P) = \tilde{\mu}_P \circ \bar{y} = \tilde{y} : K^* \rightarrow Q$. By the definition of $\tilde{y}$ this means that $(\bar{y}^* \otimes \id_Q)(\hat{\mu}_P) = \hat{y} \in K \otimes Q$.
By the definition of $\bar{b}$ we have $(\bar{b} \otimes \id_Q)(\hat{\mu}_P) = (\bar{b}_0 \otimes \id_Q)(\hat{\mu}_P) - ((i \circ \bar{y}^*) \otimes \id_Q)(\hat{\mu}_P) = \hat{x} - (i \otimes \id_Q) \circ (\bar{y}^* \otimes \id_Q)(\hat{\mu}_P) = \hat{x} - (i \otimes \id_Q)(\hat{y}) = \hat{x} - \hat{x} = 0$, as required.
\end{proof}

\begin{thm} \label{thm:real-a-b}
Let $x = (\UM; L, V_0, V_1, p_0, p_1) \in \gc_{2q+1}(Q,v,\UP,\UP)$ be an elementary proper two-sided quasi-formation such that $I_x = \id_P$. Suppose that $\mu_P : P \rightarrow Q$ is surjective. Let $I : \UP \rightarrow \UP$ be an automorphism of the form 

a) $I = \bigl[ \begin{smallmatrix} a & 0 \\ 0 & \id \end{smallmatrix} \bigr] : \fp \oplus \tp \rightarrow \fp \oplus \tp$ for some automorphism $a$ of $\UbP$; or

b) $I = \bigl[ \begin{smallmatrix} \id & 0 \\ b & \id \end{smallmatrix} \bigr] : \fp \oplus \tp \rightarrow \fp \oplus \tp$ for some homomorphism $b : \fp \rightarrow \tp$ such that $(b \otimes \id_Q)(\hat{\mu}_P) = 0 \in \tp \otimes Q$. 

Then there is a lagrangian $L'$ in $\UM$ such that $x' = (\UM; L', V_0, V_1, p_0, p_1)$ is also elementary (and proper) and $I_{x'} = I$.
\end{thm}

\begin{proof}
Write $\UM = (M, \lambda, \mu)$. Let $i_0 : \fp \rightarrow V_0$ be a splitting of the composition of surjections $V_0 \xra{p_0} P \rightarrow \fp$ and $Y_0 = \Image(i_0)$. Since $x$ is elementary, projection along $L$ is an isomorphism $V_0 \rightarrow V_1$ (see Lemma \ref{lem:pr-el-equiv} d)). Let $i_1 : \fp \rightarrow V_1$ be the composition of $i_0$ with this isomorphism and $Y_1 = \Image(i_1)$. Then the composition $\fp \xra{i_0} Y_0 \rightarrow Y_1 \xra{p_1} P$ is the restriction of $I_x = \id_P$ to $\fp$, that is, $\id_{\fp}$. In particular, $i_1 : \fp \rightarrow Y_1 \leq V_1$ is a splitting of $V_1 \xra{p_1} P \rightarrow \fp$. Since $p_0 : \UV_0 \rightarrow \UP$ and $p_1 : \UV_1 \rightarrow -\UPs$ are morphisms of extended quadratic forms, so are the isomorphisms $i_0 : \UbP \rightarrow \UY_0$ and $i_1 : -\UbPs \rightarrow \UY_1$, where $\UY_i = (Y_i, \lambda \big| _{Y_i \times Y_i}, \mu \big| _{Y_i})$.

For $i = 0,1$, let $Z_i = p_i^{-1}(\tp) \leq V_i$, equivalently, the kernel of the composition $V_i \xra{p_i} P \rightarrow \fp$. Then $V_i = Y_i \oplus Z_i$. We have $\mu(x) = \mu_P(p_i(x)) = 0$ and $\lambda(x,y) = \lambda_P(p_i(x), p_i(y)) = 0$ for every $x \in Z_i$, $y \in V_i$. Therefore $\lambda \big| _{Z_i \times Z_i} = 0$ and $\mu \big| _{Z_i} = 0$, and $\UV_i = \UY_i \oplus \UZ_i$, where $\UZ_i = (Z_i,0,0)$. Since $\UbP$ is nonsingular, so is $-\UbPs$, and hence $\UY_i$ for both $i$. This implies that $Z_i = V_i \cap V_i^{\perp} = V_i \cap V_{1-i}$, in particular $Z_0 = Z_1$, and we will denote it by $Z$. It also follows that $Y_0 \cap Y_1 = \{ 0 \}$. 

Let $N = Y_0\oplus Y_1$ and $\UN = (N, \lambda \big| _{N \times N}, \mu \big| _{N}) = \UY_0 \oplus \UY_1$, then $i_0 \oplus i_1 : \UbP \oplus (-\UbPs) \rightarrow \UN$ is an isomorphism. Since $\UN$ is nonsingular, $M = N \oplus N^{\perp}$ (see \cite[Lemma (3.1)]{milnor-husemoller73}). Let $R = N^{\perp}$, and $\UR = (R, \lambda \big| _{R \times R}, \mu \big| _R)$, then $\UM = \UN \oplus \UR$, and it follows that $\UR$ is also nonsingular. 

We have $Z = V_0 \cap V_1 = V_1^{\perp} \cap V_0^{\perp} \leq Y_1^{\perp} \cap Y_0^{\perp} = (Y_0 \oplus Y_1)^{\perp} = R$. We already saw that $\lambda \big| _{Z \times Z} = 0$ and $\mu \big| _Z = 0$. Also, $Z$ is a direct summand in $V_i$, hence in $M$, hence in $R$. Finally, $\rk Z = \rk V_i - \rk Y_i = \frac{1}{2}(\rk \UM - \rk \UN) = \frac{1}{2} \rk \UR$. Therefore $Z$ is a lagrangian in $\UR$, and hence $\UR$ is metabolic. 

Since $x$ is proper, $\Ker p_0 = \Ker p_1 \leq V_0 \cap V_1 = Z$. Moreover, the projection along $L$ restricts to $\id_Z : Z \rightarrow Z$, and hence induces $\id : Z / \Ker p_0 \rightarrow Z / \Ker p_1$. Thus $I_x \big| _{\tp} = \id_{\tp}$ is the composition $\tp \rightarrow Z / \Ker p_0 = Z / \Ker p_1 \rightarrow \tp$ of the maps induced by $p_0$ and $p_1$, and this implies that $p_0 \big| _Z = p_1 \big| _Z$. Because of this, if $L'$ is any lagrangian in $\UM$ such that $x' = (\UM; L', V_0, V_1, p_0, p_1)$ is elementary, then $I_{x'} \big| _{\tp} = \id_{\tp}$ (as the projection along $L'$ also restricts to $\id_Z$).

Let $\Delta = (i_1 - i_0)(\fp) \leq i_0(\fp) \oplus i_1(\fp) = N$. We have $\Delta \leq L$, because $i_1$ was defined using the projection along $L$. Let $C = L \cap Y_0^{\perp}$, it is the kernel of the map $L \rightarrow Y_0^*$ determined by $\ad \lambda$. This map restricts to an isomorphism $\Delta \cong Y_0^*$, because $\bar{\lambda}_P$ is nonsingular, $i_0 : \UbP \rightarrow \UY_0$ is an isomorphism, and $i_1(\fp) =  Y_1 \leq V_1$ is orthogonal to $i_0(\fp) = Y_0 \leq V_0$. It follows that $L = \Delta \oplus C$. 

Since $C$ is orthogonal to both $\Delta$ (because $C, \Delta \leq L = L^{\perp}$) and $Y_0 = i_0(\fp)$, it is also orthogonal to $Y_1 = i_1(\fp)$. Thus $C \leq R$. Since $L = \Delta \oplus C$ and $M = N \oplus R$, we get that $\Delta = L \cap N$ and $C = L \cap R$. Further, $x$ is elementary, so $M = V_0 \oplus L = Y_0 \oplus Z \oplus \Delta \oplus C$, and $Y_0 \oplus \Delta = N$ and $Z \oplus C \leq R$, which implies that $Z \oplus C = R$. It follows that $\mu \big| _R = 0$, so by Lemma \ref{lem:hyp} $\UR$ is hyperbolic.

a) Let $i_1' = i_1 \circ a : \fp \rightarrow Y_1$, then $i_1'$ is another isomorphism $-\UbPs \rightarrow \UY_1$. 
Let $\Delta' = (i_1' - i_0)(\fp)$. Equivalently, it is the image of the anti-diagonal $\Delta^*_{\id_{\fp}} \leq \fp \oplus \fp$ under the isomorphism $i_0 \oplus i_1' : \UbP \oplus (-\UbPs) \rightarrow \UN$. It is a lagrangian in $\UN$ (because $\Delta^*_{\id_{\fp}}$ is a lagrangian in $\UbP \oplus (-\UbPs)$, see \cite[Lemma 2.18]{csn-qfc26}), and $Y_0 \oplus \Delta' = N$. Let $L' = \Delta' \oplus C$, it is a lagrangian in $\UM$ and $V_0 \oplus L' = M$, so $x'$ is elementary. The projection along $L'$ restricts to $i_1' \circ i_0^{-1} : Y_0 \rightarrow Y_1'$, hence $I_{x'} \big| _{\fp} = p_1 \circ i_1' \circ i_0^{-1} \circ i_0 = p_1 \circ i_1' = p_1 \circ i_1 \circ a = \id_{\fp} \circ a = a$. Therefore $I_{x'} = I$.

b) By Lemma \ref{lem:b-lift} there is a homomorphism $\bar{b} : \fp \rightarrow Z$ such that $p \circ \bar{b} = b$ and $(\bar{b} \otimes \id_Q)(\hat{\mu}_P) = 0 \in Z \otimes Q$. Let $i_1' = i_1 + \bar{b} : \fp \rightarrow Y_1 \oplus Z = V_1$ and $Y_1' = \Image(i_1')$, then $V_1 = Y_1 \oplus Z = Y_1' \oplus Z$. 

Let $N' = Y_0\oplus Y_1'$ and $\UNp = (N', \lambda \big| _{N' \times N'}, \mu \big| _{N'})$, then $i_0 \oplus i_1' : \UbP \oplus (-\UbPs) \rightarrow \UNp$ is an isomorphism. Since $\UNp$ is nonsingular, $M = N' \oplus (N')^{\perp}$. Let $R' = (N')^{\perp}$, and $\URp = (R', \lambda \big| _{R' \times R'}, \mu \big| _{R'})$, then $\UM = \UNp \oplus \URp$, and it follows that $\URp$ is also nonsingular. 

Let $b' : C \rightarrow Y_1$ be the composition 
\[
C \xra{\ad \lambda} Z^* \xra{\bar{b}^*} \fp^* \xra{(\ad \bar{\lambda}_P)^{-1}} \fp \xra{i_1} Y_1
\]
where $C \cong Z^*$ via $\ad \lambda$, because $Z$ and $C$ are complementary lagrangians in the hyperbolic form $\UR$, and $\ad \bar{\lambda}_P$ is an isomorphism $\fp \cong \fp^*$, because $\UbP$ is nonsingular. Let $C' = (\id_C + b')(C) \leq C \oplus Y_1$, equivalently, the graph of $b'$. In particular, $C \oplus Y_1 = C' \oplus Y_1$

First we prove that $R' = Z \oplus C'$. Using that $C \oplus Y_1 = C' \oplus Y_1$ and $Y_1 \oplus Z = Y_1' \oplus Z$, we have $M = N \oplus R = Y_0 \oplus Y_1 \oplus Z \oplus C = Y_0 \oplus Y_1' \oplus Z \oplus C' = N' \oplus (Z \oplus C')$. Since also $M = N' \oplus R'$, it is enough to show that $Z \oplus C' \leq R' = (Y_0 \oplus Y_1')^{\perp}$. We have $Z = V_0 \cap V_1 = V_1^{\perp} \cap V_0^{\perp} \leq (Y_1')^{\perp} \cap Y_0^{\perp} = (Y_1' \oplus Y_0)^{\perp}$ and $C' \leq C \oplus Y_1 \leq Y_0^{\perp}$, so it remains to prove that $C' \leq (Y_1')^{\perp}$. Take any $x \in C$ and $y \in \fp$, so that $x + b'(x) \in C'$ and $i_1(y) + \bar{b}(y) \in Y_1'$. Since $C, Z \leq Y_1^{\perp}$, we have $\lambda(x,i_1(y)) = \lambda(b'(x),\bar{b}(y)) = 0$. Further, $\lambda(b'(x),i_1(y)) = \lambda(i_1 \circ (\ad \bar{\lambda}_P)^{-1} \circ \bar{b}^* \circ (\ad \lambda)(x),i_1(y)) = -\bar{\lambda}_P((\ad \bar{\lambda}_P)^{-1} \circ \bar{b}^* \circ (\ad \lambda)(x),y) = -(\bar{b}^* \circ (\ad \lambda)(x))(y) = -(\ad \lambda(x))(\bar{b}(y)) = -\lambda(x,\bar{b}(y))$. Therefore $\lambda(x + b'(x),i_1(y) + \bar{b}(y)) = 0$, showing that $C' \leq (Y_1')^{\perp}$, as required. 

Next we prove that $\mu \big| _{R'} = 0$. Since $\mu \big| _Z = 0$, this is equivalent to $\mu \big| _{C'} = 0$. Suppose that $\varphi \in Z^*$. Since $(\bar{b} \otimes \id_Q)(\hat{\mu}_P) = 0$, we have $((\varphi \circ \bar{b}) \otimes \id_Q)(\hat{\mu}_P) = (\varphi \otimes \id_Q) \circ (\bar{b} \otimes \id_Q)(\hat{\mu}_P) = 0$. By the definition of $\hat{\mu}_P$, this means that $\tilde{\mu}_P(\varphi \circ \bar{b}) = 0$. Hence $0 = \tilde{\mu}_P(\varphi \circ \bar{b}) = \tilde{\mu}_P \circ \bar{b}^*(\varphi) = \bar{\mu}_P \circ (\ad \bar{\lambda}_P)^{-1} \circ \bar{b}^*(\varphi) = \mu \circ i_1 \circ (\ad \bar{\lambda}_P)^{-1} \circ \bar{b}^*(\varphi)$, where we used that the correspondence between $\bar{\mu}_P$ and $\tilde{\mu}_P$ is given by $\ad \bar{\lambda}_P$, and that $i_1$ is a morphism of extended quadratic forms. Since this holds for any $\varphi \in Z^* \cong C$, we conclude that $\Image(b') = \Image(i_1 \circ (\ad \bar{\lambda}_P)^{-1} \circ \bar{b}^*) \leq \Ker(\mu)$. As we also have $\Image(\id_C) = C \leq \Ker(\mu)$, we get that $C' = \Image(\id_C + b') \leq \Ker(\mu)$.

Since $Z$ is a summand in $R'$ and $\rk Z = \frac{1}{2} \rk \UR = \frac{1}{2} \rk \URp$, it is a lagrangian in $\URp$, so $\URp$ is metabolic. By Lemma \ref{lem:hyp}, $\URp$ is hyperbolic. By \cite[Corollary 2.25]{csn-qfc26}, there is an isomorphism $J : \UH_{2k} \rightarrow\URp$ such that $J(\{ 0 \} \times \Z^k) = Z$, where $2k = \rk \URp$. Let $C'' = J(\Z^k \times \{ 0 \})$, it is a lagrangian complement to $Z$ in $R'$. 

Let $\Delta' = (i_1' - i_0)(\fp) = (i_0 \oplus i_1')(\Delta^*_{\id_{\fp}})$. It is a lagrangian in $\UNp$, and $Y_0 \oplus \Delta' = N'$.

Let $L' = \Delta' \oplus C''$. Since $\Delta'$ and $C''$ are lagrangians in $\UNp$ and $\URp$ respectively, $L'$ is a lagrangian in $\UM$. Moreover, $M = N' \oplus R' = Y_0 \oplus \Delta' \oplus Z \oplus C'' = V_0 \oplus L'$, so $x'$ is elementary. Since $\Delta' = (i_1' - i_0)(\fp) \leq L'$, the projection along $L'$ restricts to $i_1' \circ i_0^{-1} : Y_0 \rightarrow Y_1'$, hence $I_{x'} \big| _{\fp} = p_1 \circ i_1' \circ i_0^{-1} \circ i_0 = p_1 \circ i_1' = p_1 \circ i_1 + p_1 \circ \bar{b} = \id_{\fp} + b$. Therefore $I_{x'} = I$. 
\end{proof}

\subsection{Realisation by h-cobordisms}

Let $B$ be a simply-connected space, $\xi$ a stable vector bundle over $B$ and $q \geq 2$ an even integer. Let $M$ be a closed oriented $2q$-manifold with a normal $(q{-}1)$-smoothing $f : M \rightarrow B$ over $(B,\xi)$. 
Write $Q = H_q(B)$.
 
As before, let $\tp = \Tor H_q(M)$ and $\fp = H_q(M) / \Tor H_q(M)$, and fix a splitting $H_q(M) = \fp \oplus \tp$.
If $\mu_M \big| _{\tp} = 0$ (which automatically holds when $Q = H_q(B)$ is free), then $E_q(M,f)$ has a free reduction $\overline{E_q(M,f)}$. We define the elements $\tilde{\mu}_M \in \Hom(\fp^*, Q)$ and $\hat{\mu}_M \in \fp \otimes Q$ corresponding to $\mu_M \big| _{\fp} \in \Hom(\fp, Q)$ as in Definition \ref{def:can}.

\begin{lem} \label{lem:real-a-b-hcob}
Suppose that $H_q(B)$ is free. Let $I$ be an automorphism of $E_q(M, f)$. Assume that

a) $I = \bigl[ \begin{smallmatrix} a & 0 \\ 0 & \id \end{smallmatrix} \bigr] : \fp \oplus \tp \rightarrow \fp \oplus \tp$ for some automorphism $a$ of $\overline{E_q(M,f)}$; or

b) $I = \bigl[ \begin{smallmatrix} \id & 0 \\ b & \id \end{smallmatrix} \bigr] : \fp \oplus \tp \rightarrow \fp \oplus \tp$ for some homomorphism $b : \fp \rightarrow \tp$ such that $(b \otimes \id_Q)(\hat{\mu}_M) = 0 \in \tp \otimes Q$.

Then there is a normal bordism $F' : W' \rightarrow B$ between $f$ and itself, normally bordant to the trivial normal bordism, such that $W'$ is an h-cobordism and $I_{W'} = I$.
\end{lem}

\begin{proof}
Let $F : W = M \times [0,1] \rightarrow B$ be the trivial normal bordism between $f$ and itself. Let $\theta^{f,f}_{W,F} = (\UM; L, V_0, V_1, p_0, p_1) \in \s^T_{2q+1}(B,\xi,f,f)$ and define its free reduction $\bar{\theta}^{f,f}_{W,F} = (\UbM; \bar{L}; \bar{V}_0, \bar{V}_1, \bar{p}_0, \bar{p}_1) \in \s_{2q+1}(B,\xi,f,f)$. By Propositions \ref{prop:obstr-proper} and \ref{prop:hcob-elem}, $\theta^{f,f}_{W,F}$ is both proper and elementary, and $I_{\theta^{f,f}_{W,F}} = I_W = \id$, using that $W$ is the trivial cobordism. By Lemmas \ref{lem:pr-el-equiv} and \ref{lem:I-stab-free} $\bar{\theta}^{f,f}_{W,F}$ is also proper and elementary, and $I_{\bar{\theta}^{f,f}_{W,F}} = I_{\theta^{f,f}_{W,F}} = \id$.

We can apply Theorem \ref{thm:real-a-b} to $\UP = E_q(M, f)$, $x = \bar{\theta}^{f,f}_{W,F}$ and $I$, because $\mu_P = \mu_M$ is surjective, because $f$ is $q$-connected. We get that there is a lagrangian $\bar{L}'$ in $\UbM$ such that $\hat{\theta}' = (\UbM; \bar{L}'; \bar{V}_0, \bar{V}_1, \bar{p}_0, \bar{p}_1)$ is elementary and $I_{\hat{\theta}'} = I$. By Theorem \ref{thm:jacobi}
\[
[\UbM; \bar{L}; \bar{L}', \bar{L}', 0, 0] \oplus [\UbM; \bar{L}'; \bar{V}_0, \bar{V}_1, \bar{p}_0, \bar{p}_1] = [\UbM; \bar{L}; \bar{V}_0, \bar{V}_1, \bar{p}_0, \bar{p}_1] \in \ell_{2q+1}(B,\xi,f,f)
\]
Here $[\UbM; \bar{L}; \bar{L}', \bar{L}', 0, 0] \in L_{2q+1}(B,\xi,f,f) \cong 0$ using that $Q$ is free (see Proposition \ref{prop:L-forget}), so $\hat{\theta}' \sim \bar{\theta}^{f,f}_{W,F}$.

By Proposition \ref{prop:red-surj} b) there is a $\theta' \in \s^T_{2q+1}(B,\xi,f,f)$ such that $\theta^{f,f}_{W,F} \sim \theta'$ and $\bar{\theta}' \si \hat{\theta}'$. This implies that $\theta'$ is proper and $\bar{\theta}'$ is elementary, and by Lemma \ref{lem:pr-el-equiv} b) $\theta'$ is elementary too. By Lemma \ref{lem:I-stab-free} $I_{\theta'} = I_{\bar{\theta}'} = I_{\hat{\theta}'} = I$.

By Theorem \ref{thm:theta-equiv-real} there is a $q$-connected normal bordism $F' : W' \rightarrow B$, normally bordant to $F$, such that $\theta^{f,f}_{W',F'} \si \theta'$. By Proposition \ref{prop:hcob-elem} $W'$ is an h-cobordism and $I_{W'} = I_{\theta'} = I$. 
\end{proof}

\begin{thm} \label{thm:real-aut-diff}
Suppose that $H_q(B)$ is free. Let $I = \bigl[ \begin{smallmatrix} a & 0 \\ b & c \end{smallmatrix} \bigr] : \fp \oplus \tp \rightarrow \fp \oplus \tp$ be an automorphism of $E_q(M,f)$. Then the following are equivalent:
\begin{compactenum}
\item There is a normal bordism $F' : W' \rightarrow B$ between $f$ and itself, normally bordant to the trivial normal bordism, such that $W'$ is an h-cobordism and $I_{W'} = I$.
\item There is a normal bordism $F' : W' \rightarrow B$ between $f$ and itself, such that $W'$ is an h-cobordism and $I_{W'} = I$.
\item $I$ extends to an automorphism $\II$ of $\EE_q(M,f)$. 
\item $(b \otimes \id_Q)(\hat{\mu}_M) = 0$ and $c = \id_{\tp}$.
\end{compactenum}
\end{thm}

\begin{proof}
(1) $\Rightarrow$ (2) is trivial, while (2) $\Rightarrow$ (3) follows from Proposition \ref{prop:ind-sys}.

(3) $\Rightarrow$ (4):
Note that $H_{q-1}(f)$ is an isomorphism, because $f$ is $q$-connected. By Proposition \ref{prop:c-id} $c = \id_{\tp}$. Thus $I = I_1 \circ I_2$ for $I_1 = \bigl[ \begin{smallmatrix} a & 0 \\ 0 & \id \end{smallmatrix} \bigr]$ and $I_2 = \bigl[ \begin{smallmatrix} \id & 0 \\ b & \id \end{smallmatrix} \bigr]$. By Lemma \ref{lem:real-a-b-hcob} a) there is a normal bordism $F_1 : W_1 \rightarrow B$ with $W_1$ an h-cobordism and $I_{W_1} = I_1$. By Proposition \ref{prop:ind-sys} $I_1$ extends to an automorphism $\II_1$ of $\EE_q(M,f)$. Hence $I_2 = I_1^{-1} \circ I$ extends to an automorphism of $\EE_q(M,f)$, namely $\II_1^{-1} \circ \II$. By Theorem \ref{thm:b-mu} $(b \otimes \id_Q)(\hat{\mu}_M) = 0$. 

(4) $\Rightarrow$ (1): We write $I = I_1 \circ I_2$ as above. By Lemma \ref{lem:real-a-b-hcob} there is a normal bordism $F_j : W_j \rightarrow B$, normally bordant to the trivial normal bordism, with $W_j$ an h-cobordism and $I_{W_j} = I_j$, for $j=1,2$. We construct the normal bordism $F' : W' \rightarrow B$ by gluing $F_1$ and $F_2$, then $W'$ is also an h-cobordism and $I_{W'} = I_1 \circ I_2 = I$. Since $F_1$ and $F_2$ are normally bordant to the trivial normal bordism, the same is true for $F'$. 
\end{proof}

\section{Realising isomorphisms when $H_{q-1}(B)$ contains torsion} \label{s:real-tor}

\begin{proof}[Proof of Theorem \ref{thm:main-torsion}]
The existence of an $\II$ is necessary by Proposition \ref{prop:ind-sys}.

Now assume that an $\II$ exists. By the special case of the Q-form conjecture \cite[Theorem 1.2]{csn-qfc26} the existence of $I$ implies that there is a normal bordism $F_1 : W_1 \rightarrow B$ between $f_0$ and $f_1$, normally bordant to $F_0$, such that $W_1$ is an h-cobordism. The homotopy equivalence $M_0 \rightarrow M_1$ determined by $W_1$ induces an isomorphism $\JJ = (J, \{ J_m \} ) : \EE_q(M_0, f_0) \rightarrow \EE_q(M_1, f_1)$. Then $J^{-1} \circ I$ is an automorphism of $E_q(M_0, f_0)$, which extends to the automorphism $\JJ^{-1} \circ \II$ of $\EE_q(M_0, f_0)$. So by Theorem \ref{thm:real-aut-diff} there is a normal bordism $F_2 : W_2 \rightarrow B$ between $f_0$ and itself, normally bordant to the trivial normal bordism, such that $W_2$ is an h-cobordism and $I_{W_2} = J^{-1} \circ I$. We construct the normal bordism $F' : W' \rightarrow B$ between $f_0$ and $f_1$ by gluing $F_1$ and $F_2$, then $W'$ is also an h-cobordism and $I_{W'} = J \circ J^{-1} \circ I = I$. Since $F_1$ and $F_2$ are normally bordant to $F_0$ and the trivial normal bordism respectively, $F'$ is also normally bordant to $F_0$. 
\end{proof}

\section{Diffeomorphisms of complete intersections} \label{s:ci}

We will prove Theorem \ref{thm:main-ci}. First we recall a few facts about complete intersections. For $\dd = \left\{ d_1, d_2, \ldots , d_k \right\}$ let $X_q(\dd)$ be the complete intersection that is the transverse intersection of hypersurfaces of degrees $d_1, d_2, \ldots , d_k$ in ${\C}P^{q+k}$. It follows from the Lefschetz Hyperplane Theorem that the inclusion $i : X_q(\dd) \rightarrow {\C}P^{q+k}$ is $q$-connected. Moreover, $i$ is covered by a bundle map $\nu_{X_q(\dd)} \rightarrow \xi_q(\dd) = -(q+k+1)\gamma \oplus \gamma^{\otimes d_1} \oplus \ldots \oplus \gamma^{\otimes d_k}$, where $\gamma$ is the conjugate of the tautological complex line bundle over ${\C}P^{q+k}$, making it a normal $(q{-}1)$-smoothing. Let $y \in H^2({\C}P^{q+k}) \cong \Z$ be a generator and let $x = i^*(y) \in H^2(X_q(\dd))$.

\begin{lem} \label{lem:ci-conj}
There is a diffeomorphism $g$ of $X_q(\dd)$ such that $g^*(x) = -x$. 
\end{lem}

\begin{proof}
Fix a representative $X \subset {\C}P^{q+k}$ of $X_q(\dd)$, given by homogeneous polynomials $p_1, p_2, \ldots , p_k \in \C[x_0, x_1, \ldots, x_{q+k}]$ with degrees $d_1, d_2, \ldots , d_k$. Then the complex conjugate polynomials $\bar{p}_1, \bar{p}_2, \ldots , \bar{p}_k$ define another representative, which we will denote by $X'$. Let $i' : X' \rightarrow {\C}P^{q+k}$ be its inclusion, and $x' = (i')^*(y) \in H^2(X')$. By Thom's argument there is a path between $(p_1, p_2, \ldots , p_k)$ and $(\bar{p}_1, \bar{p}_2, \ldots , \bar{p}_k)$ in the space of tuples of polynomials that define representatives of $X_q(\dd)$. This path gives an isotopy between $i$ and an embedding $X \rightarrow {\C}P^{q+k}$ whose image is $X'$, that is, a diffeomorphism $h : X \rightarrow X'$. Since $i$ and $i' \circ h : X \rightarrow {\C}P^{q+k}$ are isotopic, we have $h^*(x') = h^* \circ (i')^*(y) = i^*(y) = x$. 

Complex conjugation in all $q+k+1$ coordinates defines a diffeomorphism $J : {\C}P^{q+k} \rightarrow {\C}P^{q+k}$ such that $J^*(y) = -y$. It restricts to a diffeomorphism $j = J \big| _X : X \rightarrow X'$ such that $i' \circ j = J \circ i$. Hence $j^*(x') = j^* \circ (i')^*(y) = i^* \circ J^*(y) = i^*(-y) = -x$. Let $g = h^{-1} \circ j : X \rightarrow X$, then $g^*(x) = j^* \circ (h^*)^{-1}(x) = j^*(x') = -x$. 
\end{proof}

\begin{proof}[Proof of Theorem \ref{thm:main-ci}]
Since $q > 2$ and $i$ is $q$-connected, it induces an isomorphism $H^2(X_q(\dd)) \cong H^2({\C}P^{q+k}) \cong \Z$. So any automorphism of $H^*(X_q(\dd))$ sends $x$ to either $x$ or $-x$. By Lemma \ref{lem:ci-conj} there is a diffeomorphism of $X_q(\dd)$ that sends $x$ to $-x$. Therefore it is enough to prove that every automorphism of $H^*(X_q(\dd))$ that fixes $x$ can be realised by a diffeomorphism. 

Using that $i$ is $q$-connected, that $\left< y^q, i_*([X_q(\dd)]) \right> = d_1 d_2 \ldots d_k$, and Poincar\'e-duality, we see that $H^j(X_q(\dd)) \cong \Z$ for $0 \leq j \leq 2q$ even, $j \neq q$, and $x^{j/2} \neq 0 \in H^j(X_q(\dd))$. Moreover, $H^j(X_q(\dd))$ is trivial if $j$ is odd, and free if $j = q$. It follows that any automorphism of $H^*(X_q(\dd))$ that fixes $x$ is the identity on $H^j(X_q(\dd))$ for all $j \neq q$, in particular for $j=2q$, and therefore it preserves the cup product pairing $H^q(X_q(\dd)) \times H^q(X_q(\dd)) \rightarrow \Z$.

Now let $\alpha$ be an automorphism of $H^*(X_q(\dd))$ that fixes $x$. Since $H^q(X_q(\dd))$ is free, $H_q(X_q(\dd))$ is its ($\Hom$-)dual. Let $A : H_q(X_q(\dd)) \rightarrow H_q(X_q(\dd))$ be the dual of $\alpha \big| _{H^q(X_q(\dd))}$. Since $\alpha$ preserves the cup product, $A$ preserves the intersection form $\lambda_{X_q(\dd)}$. Moreover, a diffeomorphism of $X_q(\dd)$ induces $\alpha$ on $H^*(X_q(\dd))$ if and only if it preserves $x$ and induces $A$ on $H_q(X_q(\dd))$.

For any $z \in H_q(X_q(\dd))$ we have $\left< y^{q/2}, i_* \circ A(z) \right> = \left< x^{q/2}, A(z) \right> = \left< \alpha(x^{q/2}), A(z) \right> = \left< x^{q/2}, z \right> = \left< y^{q/2}, i_*(z) \right>$, where we used that $\alpha$ fixes $x^{q/2}$, and $A$ is its dual. Since $\left< y^{q/2}, \cdot \right> : H_q({\C}P^{q+k}) \rightarrow \Z$ is an isomorphism, $i_* \circ A(z) = i_*(z)$. This means that $A$ is an automorphism of the Q-form $E_q(X_q(\dd),i) = (H_q(X_q(\dd)), \lambda_{X_q(\dd)}, i_*)$.

By Theorem \ref{thm:main-free} (using that $H_q({\C}P^{q+k}) \cong \Z$ and $H_{q-1}({\C}P^{q+k}) \cong 0$ are free) there is a diffeomorphism $g$ of $X_q(\dd)$ inducing $A$. Moreover, $g$ is a diffeomorphism over ${\C}P^{q+k}$, ie.\ $i \circ g \simeq i$, so $g^*(x) = g^* \circ i^*(y) = i^*(y) = x$. By our earlier observations, $g$ induces $\alpha$.
\end{proof}

\bibliographystyle{amsinitial}
\bibliography{ref}

\end{document}